\documentclass[graybox,envcountsect,envcountsame]{svmult}

\usepackage{mathptmx}       
\usepackage{helvet}         
\usepackage{courier}        
\usepackage{type1cm}        
\usepackage{makeidx}         
\usepackage{graphicx}        
\usepackage{multicol}        
\usepackage[bottom]{footmisc}

\makeindex             

\usepackage{mathtools}
\usepackage{amssymb}
\usepackage{enumerate}

\newcommand\bernkopfMelenkNorm[1]{\left\lVert#1\right\rVert}

\newtheorem{assumption}[theorem]{Assumption}

\begin{document}

\title*{Analysis of the $hp$-version of a first order system least squares method for the Helmholtz equation}
\titlerunning{Analysis of the $hp$-FOSLS method for the Helmholtz equation}
\author{M. Bernkopf and J.M. Melenk}
\institute{
Maximilian Bernkopf
\at Technische Universit\"at Wien, Institute for Analysis and Scientific Computing, Wiedner Hauptstrasse 8-10, A-1040 Vienna, \email{maximilian.bernkopf@tuwien.ac.at}
\and
Jens Markus Melenk
\at Technische Universit\"at Wien, Institute for Analysis and Scientific Computing, Wiedner Hauptstrasse 8-10, A-1040 Vienna, \email{melenk@tuwien.ac.at}
}
%
%
\maketitle

\abstract*{
Extending the wavenumber-explicit analysis of~\cite{chen-qiu17}, we analyze the $L^2$-convergence
of a least squares method for the Helmholtz equation with wavenumber $k$. For domains with
an analytic boundary, we obtain improved rates
in the mesh size $h$ and the polynomial degree $p$ under the scale resolution condition that $hk/p$
is sufficiently small and $p/\log k$ is sufficiently large.
}

\abstract{
Extending the wavenumber-explicit analysis of~\cite{chen-qiu17}, we analyze the $L^2$-convergence
of a least squares method for the Helmholtz equation with wavenumber $k$. For domains with
an analytic boundary, we obtain improved rates
in the mesh size $h$ and the polynomial degree $p$ under the scale resolution condition that $hk/p$
is sufficiently small and $p/\log k$ is sufficiently large.
}

\section{Introduction}\label{bernkopf_melenk_section:introduction}

We consider the following Helmholtz problem:
\begin{equation}\label{bernkopf_melenk_eq:helmholtz_boundary_value_problem}
	\begin{alignedat}{2}
		- \Delta u - k^2 u &= f \quad   &&\text{in } \Omega,\\
		\partial_n u - i k u &= g       &&\text{on } \partial \Omega,
	\end{alignedat}
\end{equation}
where $k \ge k_0 > 0$ is real.
For large $k$, the numerical solution of~\eqref{bernkopf_melenk_eq:helmholtz_boundary_value_problem} is challenging due to the requirement to resolve the oscillatory nature of the solution.
A second challenge arises in classical, $H^1$-conforming discretizations of~\eqref{bernkopf_melenk_eq:helmholtz_boundary_value_problem}
from the fact that the Galerkin method is \emph{not} an energy projection, and a meaningful approximation is only
obtained under more stringent conditions on the mesh size $h$ and the polynomial degree $p$ than purely approximation
theoretical considerations suggest.
This shortcoming has been analyzed in the literature.
In particular,
as discussed in more detail in~\cite{melenk-sauter11,esterhazy-melenk12},
the analyses~\cite{ihlenburg98,ihlenburg-babuska95,ihlenburg-babuska97,ainsworth04,melenk-sauter10,melenk-sauter11,esterhazy-melenk12}
show that high order methods are much better suited for the high-frequency
case of large $k$ than low order methods.
Alternatives to the classical Galerkin methods that are still based
on high order methods include stabilized methods for Helmholtz~\cite{feng-wu09,feng-wu11,feng-xing13,zhu-wu13},
hybridizable methods~\cite{chen-lu-xu13}, least-squares type methods~\cite{chen-qiu17, lee-manteuffel-mccormick-ruge00} and
Discontinuous Petrov Galerkin methods,~\cite{petrides-demkowicz17,demkowicz-gopalakrishnan-muga-zitelli12}.
An attractive feature of least squares type methods is that the resulting linear system is always solvable and
that they feature quasi-optimality, albeit in some nonstandard residual norms.
In the present paper,
we show for the least squares method~\eqref{bernkopf_melenk_eq:fosls_method_discrete} an {\sl a priori} estimate in the more tractable $L^2(\Omega)$-norm under the
scale resolution condition~\eqref{bernkopf_melenk_eq:scale_resolution_condition}. For that,
we closely follow~\cite{chen-qiu17}.
Our key refinement over~\cite{chen-qiu17} is an improved regularity estimate for the solution of a suitable dual problem
(cf.\ Lemma~\ref{bernkopf_melenk_lemma:duality_argument} vs.~\cite[Lemma~5.1]{chen-qiu17}) that allows us to establish the improved
$p$-dependence in the $L^2(\Omega)$-error estimate (cf.\ Theorem~\ref{bernkopf_melenk_theorem:a_priori_estimate} vs.~\cite[Thm.~2.5]{chen-qiu17}).
As a tool, which is of independent interest, we develop approximation operators in Raviart-Thomas and Brezzi-Douglas-Marini spaces
with optimal (in $h$ and $p$) approximation rates simultaneously in $L^2(\Omega)$ and $\pmb{H}(\operatorname{div}, \Omega)$.

Throughout this paper, if not otherwise stated, we assume the following:
\begin{assumption}\label{bernkopf_melenk_assumption:general_assumptions}
	In spatial dimension $d = 2,3$ the bounded Lipschitz domain $\Omega \subset \mathbb{R}^d$
        has an analytic boundary.  The wavenumber $k$ satisfies $ k \geq k_0 > 0 $.
	Furthermore $f \in L^2(\Omega)$ and $g \in L^2(\partial \Omega)$.
\end{assumption}

\begin{remark}\label{bernkopf_melenk_remark:analytic_boundary_gives_stability_estimate}
	\smartqed
	Under Assumption~\ref{bernkopf_melenk_assumption:general_assumptions} we may apply~\cite[Thm.~1.8]{baskin-spence-wunsch16}
	to conclude that the solution $u \in H^1(\Omega)$ satisfies the {\sl a priori} bound
	\begin{equation}
\label{bernkopf_melenk_eq:remark:analytic_boundary_gives_stability_estimate}
		\bernkopfMelenkNorm{u}_{H^1(\Omega)} + k \bernkopfMelenkNorm{u}_{L^2(\Omega)} \leq C ( \bernkopfMelenkNorm{f}_{L^2(\Omega)} + \bernkopfMelenkNorm{g}_{L^2(\partial \Omega)} ),
	\end{equation}
	where $C>0$ is independent of $k$.
	\qed
\end{remark}

\noindent
\textbf{Notation and preliminaries}:
Boldface letters like $\pmb{V}$, $\pmb{\varphi}$ and $\pmb{\Pi}$ will be reserved for quantities having more than one spatial dimensions,
while normal letters like $W$, $u$ and $\Pi$ will be used for quantities with one spatial dimension.
The reference element will be denoted by $\widehat{K}$, whereas the physical one will just be denoted by $K$.
In a similar way, we will distinguish between objects associated with the reference element and the physical one.
A function defined on the reference element $\widehat{K}$ will therefore be denoted by $\hat{u}$,
while a function defined on the physical element $K$ will be denoted by $u$.
We will follow the same convention when it comes to operators acting on a function space.
Therefore operators acting on functions defined on $\widehat{K}$ or $K$ will be denoted by $\widehat{\Pi}$ or $\Pi$ respectively.
Generic constants will either be denoted by $C$ or hidden inside a $\lesssim$
and will be independent of the wavenumber $k$, the mesh size $h$ and the polynomial degree $p$, if not otherwise stated.

\noindent
\textbf{Outline}:
The outline of this paper is as follows.
In Section 2 we introduce the first order system least squares (FOSLS) method itself,
followed by Section 3, where we prove a refined duality argument (Lemma~\ref{bernkopf_melenk_lemma:duality_argument}), which is later used to derive an {\sl a priori} estimate (Theorem~\ref{bernkopf_melenk_theorem:a_priori_estimate}) of the method.
Key ingredients are the results of~\cite{melenk-parsania-sauter13}, where a frequency explicit splitting of the solution to~\eqref{bernkopf_melenk_eq:helmholtz_boundary_value_problem} is performed when the data has higher order Sobolev regularity.
Section 4 is concerned with the approximation properties of Raviart-Thomas and Brezzi-Douglas-Marini spaces.
We therefore follow the methodology of~\cite{melenk-sauter10} in order to construct approximation operators,
which are not only $p$-optimal and approximate simultaneous in $L^2(\Omega)$ and $H^1(\Omega)$, but also admit an elementwise construction.
Section 5 is then devoted to the {\sl a priori} estimate.
Concluding, we give numerical examples which compliment the theoretical findings and compare the method to the classical FEM in Section 6.

\section{First order system least squares method and useful results}

In the present Section we introduce the method of~\cite{chen-qiu17} and list some useful results which are used later in the paper.

\subsection{First order system least squares}

We employ the complex Hilbert spaces
\begin{equation*}
	\pmb{V} = \{ \pmb{\varphi} \in \pmb{H}(\operatorname{div}, \Omega) \colon \pmb{\varphi} \cdot \pmb{n} \in L^2( \partial \Omega) \}
\quad
  \text{ and }
\quad
  W = H^1(\Omega),
\end{equation*}
where $\pmb{V}$ is endowed with the usual graph norm and $W$ with the classical $H^1(\Omega)$-norm.
On $\pmb{V} \times W$ we introduce the bilinear form $b$ and the linear functional $F$ by
\begin{align*}
	b( (\pmb{\varphi}, u), (\pmb{\psi}, v) ) &\coloneqq (i k \pmb{\varphi} + \nabla u , i k \pmb{\psi} + \nabla v )_{\Omega} + (i k u + \nabla \cdot \pmb{\varphi} , i k v + \nabla \cdot \pmb{\psi} )_{\Omega} + \\
			& \phantom{ \coloneqq } k ( \pmb{\varphi} \cdot \pmb{n} + u, \pmb{\psi} \cdot \pmb{n} + v )_{\partial \Omega}, \\
	F((\pmb{\psi}, v)) &\coloneqq (-i k^{-1} f, i k v + \nabla \cdot \pmb{\psi})_{\Omega} + (i g, \pmb{\psi} \cdot \pmb{n} + v)_{\partial \Omega},
\end{align*}
where $(u,v)_{\Omega} = \int_{\Omega} u \overline{v} \, \D x$.
If $u \in H^1(\Omega)$ is the weak solution to~\eqref{bernkopf_melenk_eq:helmholtz_boundary_value_problem}
then the pair $(\pmb{\varphi}, u)$ with $\pmb{\varphi} = i k^{-1} \nabla u$ is in fact in $\pmb{V} \times W$
due to the assumed regularity of the data and the domain and therefore satisfies
\begin{equation}\label{bernkopf_melenk_eq:fosls_method_continuous}
  b( (\pmb{\varphi}, u), (\pmb{\psi}, v) ) = F((\pmb{\psi}, v)) ~~~ \forall (\pmb{\psi}, v) \in \pmb{V} \times W.
\end{equation}
For a given regular mesh $\mathcal{T}_h$ we consider the finite element spaces $\pmb{V}_h = \pmb{\mathrm{RT}}_p(\mathcal{T}_h) \subset \pmb{V}$ or $\pmb{V}_h = \pmb{\mathrm{BDM}}_p(\mathcal{T}_h) \subset \pmb{V}$ and $W_h = S_p(\mathcal{T}_h)\subset W$,
where $\pmb{\mathrm{RT}}_p(\mathcal{T}_h)$ denotes the Raviart-Thomas space and $\pmb{\mathrm{BDM}}_p(\mathcal{T}_h)$ the Brezzi-Douglas-Marini space; see Section~\ref{bernkopf_melenk_section:approximation_propertie_of_raviart_thomas} for further detail and definitions.
The FOSLS method is to find $(\pmb{\varphi}_h, u_h) \in \pmb{V}_h \times W_h$ such that
\begin{equation}\label{bernkopf_melenk_eq:fosls_method_discrete}
  b( (\pmb{\varphi}_h, u_h), (\pmb{\psi}_h, v_h) ) = F((\pmb{\psi}_h, v_h)) ~~~ \forall (\pmb{\psi}_h, v_h) \in \pmb{V}_h \times W_h.
\end{equation}

\begin{remark}\label{bernkopf_melenk_remark:unique_solvability}
Based on the {\sl a priori}
estimate~\eqref{bernkopf_melenk_eq:remark:analytic_boundary_gives_stability_estimate} reference
	\cite[Thm.~2.4]{chen-qiu17} asserts the existence of $C>0$ independent of $k$ such that
	\begin{equation*}
		\bernkopfMelenkNorm{\pmb{\varphi}}_{L^2(\Omega)}^2 + \bernkopfMelenkNorm{u}_{L^2(\Omega)}^2 + k \bernkopfMelenkNorm{\pmb{\varphi} \cdot \pmb{n} + u}_{L^2(\partial \Omega)}^2
		\leq C b( (\pmb{\varphi}, u), (\pmb{\varphi}, u) ), \quad \forall (\pmb{\varphi}, u) \in \pmb{V} \times W,
	\end{equation*}
	which immediately gives uniqueness.
	Together with the fact that the pair $(\pmb{\varphi}, u)$ with $\pmb{\varphi} = i k^{-1} \nabla u$ is a solution, we have unique solvability of~\eqref{bernkopf_melenk_eq:fosls_method_continuous}.
\qed
\end{remark}

\subsection{Auxiliary results}

We will need the following decomposition result for the refined duality argument in Lemma~\ref{bernkopf_melenk_lemma:duality_argument}.

\begin{proposition}[{\cite[Thm.~4.5]{melenk-parsania-sauter13}} combined with {\cite[Thm.~1.8]{baskin-spence-wunsch16}}]\label{bernkopf_melenk_proposition:higher_regularity_decomposition}
	Let $ \Omega \subset \mathbb{R}^d, d \in \left\{ 2, 3 \right\} $, be a bounded Lipschitz domain with an
analytic boundary.  Fix $s \in \mathbb{N}_0$.
  Then there exist constants $C, \gamma > 0$ independent of $k$ such that for every $f \in H^s(\Omega)$ and $g \in H^{s+1/2}(\partial \Omega)$ the solution $u = S_k(f,g)$ of~\eqref{bernkopf_melenk_eq:helmholtz_boundary_value_problem} can be written as $u = u_A + u_{H^{s+2}}$, where, for all $n \in \mathbb{N}_0$, there holds
  \begin{align}
    \bernkopfMelenkNorm{u_A}_{H^1(\Omega)} + k \bernkopfMelenkNorm{u_A}_{L^2(\Omega)}
        \leq C
              &( \bernkopfMelenkNorm{f}_{ L^2(\Omega) } + \bernkopfMelenkNorm{g}_{ H^{1/2}(\partial \Omega)} ), \\
    \bernkopfMelenkNorm{\nabla^{n+2} u_A}_{L^2(\Omega)}
        \leq C \gamma^n k^{- 1} \max\left\{ n,k \right\}^{n+2}
              &( \bernkopfMelenkNorm{f}_{L^2(\Omega)} + \bernkopfMelenkNorm{g}_{H^{1/2}(\partial \Omega)} ), \\
    \bernkopfMelenkNorm{u_{H^{s+2}}}_{H^{s+2}(\Omega)} + k^{s+2} \bernkopfMelenkNorm{u_{H^{s+2}}}_{L^2(\Omega)}
        \leq C &( \bernkopfMelenkNorm{f}_{H^s(\Omega)} + \bernkopfMelenkNorm{g}_{H^{s+1/2}(\partial \Omega)} ).
  \end{align}
\end{proposition}

\begin{remark}\label{bernkopf_melenk_remark:interpolation_higher_regulartiy_decomposition}
	\smartqed
  Interpolation between $L^2(\Omega)$ and $H^{s+2}(\Omega)$ in Proposition~\ref{bernkopf_melenk_proposition:higher_regularity_decomposition} gives estimates for lower order Sobolev norms: Since we have
  for any $v \in H^m(\Omega)$
  \begin{equation*}
    \bernkopfMelenkNorm{v}_{H^j(\Omega)} \leq
      C \bernkopfMelenkNorm{v}_{H^m(\Omega)}^{\frac{j}{m}} \bernkopfMelenkNorm{v}_{L^2(\Omega)}^{\frac{m-j}{m}}, \qquad j \in \{0, \dots, m\},
  \end{equation*}
  Proposition~\ref{bernkopf_melenk_proposition:higher_regularity_decomposition}
  implies for $j \in \{0, \dots, s+2\}$
  \begin{equation*}
    k^{s+2-j}\bernkopfMelenkNorm{u_{H^{s+2}}}_{H^j(\Omega)} \leq
      C ( \bernkopfMelenkNorm{f}_{H^s(\Omega)} + \bernkopfMelenkNorm{g}_{H^{s+1/2}(\partial \Omega)} ).
  \end{equation*}
	\qed
\end{remark}
Furthermore we often use the multiplicative trace inequality.
We remind the reader of the general form, even though we only need it in the special case $s = 1$.

\begin{proposition}[{\cite[Thm.~A.2]{melenk05}}]\label{bernkopf_melenk_theorem:multiplicative_trace_inquality}
	Let $ \Omega \subset \mathbb{R}^d $ be a Lipschitz domain and $ s \in ( 1/2, 1 ] $.
	Then there exists a constant $ C > 0 $ such that for all $ u \in H^s(\Omega) $ there holds
	\begin{equation*}
		\bernkopfMelenkNorm{u}_{L^2(\partial \Omega)}
			\leq C
			\bernkopfMelenkNorm{u}_{L^2(\Omega)}^{1-1/(2s)}
			\bernkopfMelenkNorm{u}_{H^s(\Omega)}^{1/(2s)},
	\end{equation*}
	where the left-hand side is understood in the trace sense.
\end{proposition}

\section{Duality Argument}\label{bernkopf_melenk_section:duality_argument}

We extend the results of~\cite[Lemma~5.1]{chen-qiu17}
by showing that the function $\pmb{\psi}_{H^2} \in \pmb{H}^1(\operatorname{div}, \Omega)$, constructed therein, can actually be modified to satisfy $\pmb{\psi}_{H^2} \in \pmb{H}^2(\Omega)$
and still allow for wavenumber-explicit higher order Sobolev norm estimates.

\begin{lemma}\label{bernkopf_melenk_lemma:duality_argument}
  For any $(\pmb{\varphi}, w) \in \pmb{V} \times W$ there exists $(\pmb{\psi}, v) \in \pmb{V} \times W$
such that $\bernkopfMelenkNorm{w}_{L^2(\Omega)}^2 = b( (\pmb{\varphi}, w), (\pmb{\psi}, v) )$.
  The pair $(\pmb{\psi}, v)$ admits a decomposition $\pmb{\psi} = \pmb{\psi}_A + \pmb{\psi}_{H^2}$, $v = v_A + v_{H^2}$, where
  $\pmb{\psi}_A$ and $v_A$ are analytic in $\Omega$, $\pmb{\psi}_{H^2} \in \pmb{H}^2(\Omega)$, and $v_{H^2} \in H^2(\Omega)$.
  Furthermore there exist constants $C, \gamma > 0$ independent of $k$ such that for all $n \in \mathbb{N}_0$
  \begin{align}
    \bernkopfMelenkNorm{\pmb{\psi}_A}_{H^1(\Omega)} + k \bernkopfMelenkNorm{\pmb{\psi}_A}_{L^2(\Omega)} &\leq C k \bernkopfMelenkNorm{w}_{L^2( \Omega )},                                                                  \label{bernkopf_melenk_eq:decomposition_lemma_1} \\
    \bernkopfMelenkNorm{v_A}_{H^1(\Omega)} + k \bernkopfMelenkNorm{v_A}_{L^2(\Omega)} &\leq C k \bernkopfMelenkNorm{w}_{L^2( \Omega )},                                                                                    \label{bernkopf_melenk_eq:decomposition_lemma_2} \\
    \bernkopfMelenkNorm{\nabla^{n+2} \pmb{\psi}_A}_{L^2(\Omega)} + \bernkopfMelenkNorm{\nabla^{n+2} v_A}_{L^2(\Omega)} &\leq C \gamma^n \max\left\{n, k\right\}^{n+2} \bernkopfMelenkNorm{w}_{L^2( \Omega )},              \label{bernkopf_melenk_eq:decomposition_lemma_3} \\
\!\!\!
\!\!\!
    \bernkopfMelenkNorm{\pmb{\psi}_{H^2}}_{H^2(\Omega)} + k \bernkopfMelenkNorm{\pmb{\psi}_{H^2}}_{H^1(\Omega)} + k^2 \bernkopfMelenkNorm{\pmb{\psi}_{H^2}}_{L^2(\Omega)} &\leq C \bernkopfMelenkNorm{w}_{L^2( \Omega )},     \label{bernkopf_melenk_eq:decomposition_lemma_4} \\
    \bernkopfMelenkNorm{v_{H^2}}_{H^2(\Omega)} + k \bernkopfMelenkNorm{v_{H^2}}_{H^1(\Omega)}+ k^2 \bernkopfMelenkNorm{v_{H^2}}_{L^2(\Omega )} &\leq C \bernkopfMelenkNorm{w}_{L^2( \Omega )}.                                \label{bernkopf_melenk_eq:decomposition_lemma_5}
  \end{align}
\end{lemma}


\begin{proof}
	\smartqed
  The proof follows the ideas of~\cite[Lemma~5.1]{chen-qiu17}; for the readers' convenience we recapitulate the important steps of the proof.
  The novelty over~\cite{chen-qiu17} is the ability to choose $\pmb{\psi}_{H^2} \in \pmb{H}^2(\Omega)$
	together with $\bernkopfMelenkNorm{\pmb{\psi}_{H^2}}_{H^2(\Omega)} \leq C \bernkopfMelenkNorm{w}_{L^2( \Omega )}$.

	Consider the problem
  \begin{equation*}
    \begin{alignedat}{2}
      - \Delta z - k^2 z &= w \qquad   &&\text{in } \Omega,\\
      \partial_n z + i k z &= 0        &&\text{on } \partial \Omega.
    \end{alignedat}
  \end{equation*}
  For any $\pmb{\varphi} \in \pmb{V}$ we have, using the weak formulation and integrating by parts,
  \begin{align*}
    \bernkopfMelenkNorm{w}_{L^2(\Omega)}^2
    &= ( \nabla w, \nabla z )_\Omega - k^2 ( w, z )_\Omega - i k ( w, z )_{\partial \Omega} \\
    &= ( i k \pmb{\varphi} + \nabla w, \nabla z )_\Omega - ( i k \pmb{\varphi} , \nabla z )_\Omega - k^2 ( w, z )_\Omega - i k ( w, z )_{\partial \Omega} \\
    &= ( i k \pmb{\varphi} + \nabla w, \nabla z )_\Omega + ( \nabla \cdot \pmb{\varphi} + i k w, - i k z )_\Omega + ( \pmb{\varphi} \cdot \pmb{n} + w, i k z )_{\partial \Omega}.
  \end{align*}
  Applying Proposition~\ref{bernkopf_melenk_proposition:higher_regularity_decomposition} together with Remark~\ref{bernkopf_melenk_remark:interpolation_higher_regulartiy_decomposition} we decompose $z$ into $z = z_A + z_{H^2}$ with $z_A$ analytic and $z_{H^2} \in H^2(\Omega)$.
  Furthermore we have, for all $n \in {\mathbb N}_0$,
  \begin{align}
    \bernkopfMelenkNorm{z_A}_{H^1(\Omega)} + k \bernkopfMelenkNorm{z_A}_{L^2(\Omega)}
        \leq C &\bernkopfMelenkNorm{w}_{L^2 (\Omega)}, \label{bernkopf_melenk_eq:z_analytic_part_low_derivatives} \\
    \bernkopfMelenkNorm{\nabla^{n+2} z_A}_{L^2(\Omega)}
        \leq C \gamma^n k^{-1} \max\left\{ n,k \right\}^{n+2} &\bernkopfMelenkNorm{w}_{L^2 (\Omega)}, \label{bernkopf_melenk_eq:z_analytic_part_high_derivatives} \\
    \bernkopfMelenkNorm{z_{H^2}}_{H^2(\Omega)} + k \bernkopfMelenkNorm{z_{H^2}}_{H^1(\Omega)} + k^2 \bernkopfMelenkNorm{z_{H^2}}_{L^2(\Omega)}
        \leq C &\bernkopfMelenkNorm{w}_{L^2 (\Omega)}. \label{bernkopf_melenk_eq:z_h_2_part}
  \end{align}
  Let $(\pmb{\psi}, v) \in \pmb{V} \times W$ solve
  \begin{equation*}
  	\begin{alignedat}{2}
  		i k \pmb{\psi} + \nabla v                           &= \nabla z    \qquad  &&\text{in } \Omega, \\
      i k v + \nabla \cdot \pmb{\psi}                     &= -i k z      \qquad  &&\text{in } \Omega, \\
  		k^{1/2} ( \pmb{\psi} \cdot \pmb{n} + v ) &= i k^{1/2} z \qquad  &&\text{on } \partial \Omega.
  	\end{alignedat}
  \end{equation*}
	Indeed, this system is uniquely solvable by Remark~\ref{bernkopf_melenk_remark:unique_solvability}.
  This gives the desired representation such that $\bernkopfMelenkNorm{w}_{L^2(\Omega)}^2 = b( (\pmb{\varphi}, w), (\pmb{\psi}, v) )$.
  Using the decomposition $z = z_A + z_{H^2}$ we obtain $(\pmb{\psi}, v) = (\tilde{\pmb{\psi}}_A, \tilde{v}_A) + (\tilde{\pmb{\psi}}_{H^2}, \tilde{v}_{H^2})$, where
  \\
  \\
  \noindent\begin{minipage}{.5\linewidth}
  \begin{equation*}
    \begin{alignedat}{2}
      i k \tilde{\pmb{\psi}}_A + \nabla \tilde{v}_A                           &= \nabla z_A    \qquad  &\text{in } \Omega, \\
      i k \tilde{v}_A + \nabla \cdot \tilde{\pmb{\psi}}_A                     &= -i k z_A      \qquad  &\text{in } \Omega, \\
      k^{1/2} ( \tilde{\pmb{\psi}}_A \cdot \pmb{n} + \tilde{v}_A ) &= i k^{1/2} z_A \qquad  &\text{on } \partial \Omega,
    \end{alignedat}
  \end{equation*}
  \end{minipage}%
  \begin{minipage}{.5\linewidth}
    \begin{equation*}
      \begin{alignedat}{2}
        i k \tilde{\pmb{\psi}}_{H^2} + \nabla \tilde{v}_{H^2}                           &= \nabla z_{H^2}    \qquad  &\text{in } \Omega, \\
        i k \tilde{v}_{H^2} + \nabla \cdot \tilde{\pmb{\psi}}_{H^2}                     &= -i k z_{H^2}      \qquad  &\text{in } \Omega, \\
        k^{1/2} ( \tilde{\pmb{\psi}}_{H^2} \cdot \pmb{n} + \tilde{v}_{H^2} ) &= i k^{1/2} z_{H^2} \qquad  &\text{on } \partial \Omega.
      \end{alignedat}
    \end{equation*}
  \end{minipage}
  \\
  \\
  One can immediately verify that
  \begin{equation}\label{bernkopf_melenk_eq:v_A_tilde_pde}
    \begin{alignedat}{2}
      - \Delta ( \tilde{v}_A - z_A ) - k^2 ( \tilde{v}_A - z_A )   &= 2 k^2 z_A \qquad            \qquad &&\text{in } \Omega,\\
      \partial_n ( \tilde{v}_A - z_A ) - i k ( \tilde{v}_A - z_A ) &= (1 + i) k z_A    \qquad &&\text{on } \partial \Omega,
    \end{alignedat}
  \end{equation}
  as well as
  \begin{equation}\label{bernkopf_melenk_eq:v_h_2_tilde_pde}
    \begin{alignedat}{2}
      - \Delta ( \tilde{v}_{H^2} - z_{H^2} ) - k^2 ( \tilde{v}_{H^2} - z_{H^2} )   &= 2 k^2 z_{H^2} \qquad            \qquad &&\text{in } \Omega,\\
      \partial_n ( \tilde{v}_{H^2} - z_{H^2} ) - i k ( \tilde{v}_{H^2} - z_{H^2} ) &= (1 + i) k z_{H^2}    \qquad &&\text{on } \partial \Omega.
    \end{alignedat}
  \end{equation}
  Note that the right-hand sides in equation~\eqref{bernkopf_melenk_eq:v_A_tilde_pde} are analytic.
  This fact is used in~\cite[Lemma~5.1, Lemma~4.4]{chen-qiu17} to prove the following bounds
for all $n \in {\mathbb N}_0$:
  \begin{align}
    \bernkopfMelenkNorm{\nabla^{n+2} \tilde{v}_A}_{L^2(\Omega)} &\leq C \gamma^n \max\left\{n, k\right\}^{n+2} \bernkopfMelenkNorm{w}_{L^2( \Omega )},                         \label{bernkopf_melenk_eq:v_A_tilde_high_derivatives}\\
    \bernkopfMelenkNorm{\tilde{v}_A}_{H^1(\Omega)} + k \bernkopfMelenkNorm{\tilde{v}_A}_{L^2(\Omega)} &\leq C k \bernkopfMelenkNorm{w}_{L^2( \Omega )},                           \label{bernkopf_melenk_eq:v_A_tilde_low_derivatives}\\
    \bernkopfMelenkNorm{\nabla^{n+2} \tilde{\pmb{\psi}}_A}_{L^2(\Omega)} &\leq C \gamma^n \max\left\{n, k\right\}^{n+2} \bernkopfMelenkNorm{w}_{L^2( \Omega )},                \label{bernkopf_melenk_eq:psi_A_tilde_high_derivatives}\\
    \bernkopfMelenkNorm{\tilde{\pmb{\psi}}_A}_{H^1(\Omega)} + k \bernkopfMelenkNorm{\tilde{\pmb{\psi}}_A}_{L^2(\Omega)} &\leq C k \bernkopfMelenkNorm{w}_{L^2( \Omega )}.         \label{bernkopf_melenk_eq:psi_A_tilde_low_derivatives}
  \end{align}
  Since $\tilde{v}_{H^2} - z_{H^2} = S_k( 2 k^2 z_{H^2}, (1 + i) k z_{H^2} )$, where $S_k$ denotes the solution
operator for~\eqref{bernkopf_melenk_eq:helmholtz_boundary_value_problem}, we can exploit the regularity of the right-hand sides in equation~\eqref{bernkopf_melenk_eq:v_h_2_tilde_pde}.
  Applying Proposition~\ref{bernkopf_melenk_proposition:higher_regularity_decomposition} with $s = 1$ as well as Remark~\ref{bernkopf_melenk_remark:interpolation_higher_regulartiy_decomposition}
  we decompose $\tilde{v}_{H^2} - z_{H^2} = \hat{v}_A + \hat{v}_{H^3}$, where $\hat{v}_A$ is analytic and $\hat{v}_{H^3} \in H^3(\Omega)$.
  For every $j \in \left\{0, 1, 2, 3\right\}$ we have
  \begin{align*}
    k^{3-j}\bernkopfMelenkNorm{\hat{v}_{H^3}}_{H^j(\Omega)}
    &\lesssim \bernkopfMelenkNorm{2 k^2 z_{H^2}}_{H^1(\Omega)} + \bernkopfMelenkNorm{(1 + i) k z_{H^2}}_{H^{3/2} (\partial \Omega)} \\
    &\lesssim \underbrace{k^2 \bernkopfMelenkNorm{z_{H^2}}_{H^1(\Omega)}}_{\stackrel{\eqref{bernkopf_melenk_eq:z_h_2_part}}{\lesssim} k \bernkopfMelenkNorm{w}_{L^2(\Omega)}}
     + \underbrace{k \bernkopfMelenkNorm{z_{H^2}}_{H^{3/2} (\partial \Omega)}}_{\lesssim k \bernkopfMelenkNorm{z_{H^2}}_{H^2(\Omega)} \stackrel{\eqref{bernkopf_melenk_eq:z_h_2_part}}{\lesssim} k \bernkopfMelenkNorm{w}_{L^2(\Omega)}} \\
    &\lesssim k \bernkopfMelenkNorm{w}_{L^2(\Omega)}.
  \end{align*}
  Summarizing the above we have
  \begin{equation}\label{bernkopf_melenk_eq:v_h_3_hat_low_derivatives}
    k^{-1} \bernkopfMelenkNorm{\hat{v}_{H^3}}_{H^3(\Omega)} + \bernkopfMelenkNorm{\hat{v}_{H^3}}_{H^2(\Omega)} + k \bernkopfMelenkNorm{\hat{v}_{H^3}}_{H^1(\Omega)}+ k^2 \bernkopfMelenkNorm{\hat{v}_{H^3}}_{L^2(\Omega )} \leq C \bernkopfMelenkNorm{w}_{L^2( \Omega )}.
  \end{equation}
  In order to analyze the behavior of $\hat{v}_A$ we first estimate
  \begin{equation*}
    \bernkopfMelenkNorm{2 k^2 z_{H^2}}_{L^2(\Omega)} + \bernkopfMelenkNorm{(1 + i) k z_{H^2}}_{H^{1/2} (\partial \Omega)}
      \stackrel{\eqref{bernkopf_melenk_eq:z_h_2_part}}{\lesssim} \bernkopfMelenkNorm{w}_{L^2(\Omega)}.
  \end{equation*}
  We therefore conclude, again with Proposition~\ref{bernkopf_melenk_proposition:higher_regularity_decomposition}, that
  \begin{align}
    \bernkopfMelenkNorm{\hat{v}_A}_{H^1(\Omega)} + k \bernkopfMelenkNorm{\hat{v}_A}_{L^2(\Omega)} &\leq C \bernkopfMelenkNorm{w}_{L^2( \Omega )},                 \label{bernkopf_melenk_eq:v_A_hat_low_derivatives} \\
    \bernkopfMelenkNorm{\nabla^{n+2} \hat{v}_A}_{L^2(\Omega)} &\leq C \gamma^n k^{ -1 } \max\left\{ n,k \right\}^{n+2} \bernkopfMelenkNorm{w}_{L^2( \Omega )}. \label{bernkopf_melenk_eq:v_A_hat_high_derivatives}
  \end{align}
	We turn to the final decompositions with associated norm bounds.

	\noindent
  \textbf{Final decomposition of $v$}:
  \begin{equation*}
    v = \tilde{v}_A + \tilde{v}_{H^2} = \tilde{v}_A + \underbrace{\tilde{v}_{H^2} - z_{H^2}}_{= \hat{v}_A + \hat{v}_{H^3}} + z_{H^2} = \underbrace{\tilde{v}_A + \hat{v}_A}_{\eqqcolon v_A} + \underbrace{\hat{v}_{H^3} + z_{H^2}}_{\eqqcolon v_{H^2}}.
  \end{equation*}

	\noindent
  \textbf{Verification of~\eqref{bernkopf_melenk_eq:decomposition_lemma_2}}:
  \begin{align*}
    \bernkopfMelenkNorm{v_A}_{H^1(\Omega)} + k \bernkopfMelenkNorm{v_A}_{L^2(\Omega)}
      &\leq \underbrace{\bernkopfMelenkNorm{\tilde{v}_A}_{H^1(\Omega)} + k \bernkopfMelenkNorm{\tilde{v}_A}_{L^2(\Omega)}}_{\stackrel{\eqref{bernkopf_melenk_eq:v_A_tilde_low_derivatives}}{\leq} C k \bernkopfMelenkNorm{w}_{L^2( \Omega )}}
          + \underbrace{\bernkopfMelenkNorm{\hat{v}_A}_{H^1(\Omega)} + k \bernkopfMelenkNorm{\hat{v}_A}_{L^2(\Omega)}}_{\stackrel{\eqref{bernkopf_melenk_eq:v_A_hat_low_derivatives}}{\leq} C \bernkopfMelenkNorm{w}_{L^2( \Omega )}} \\
      &\leq C k \bernkopfMelenkNorm{w}_{L^2( \Omega )}.
  \end{align*}

	\noindent
  \textbf{Verification of~\eqref{bernkopf_melenk_eq:decomposition_lemma_5}}:
  \begin{align*}
    \bernkopfMelenkNorm{v_{H^2}}_{H^2(\Omega)} &+ k \bernkopfMelenkNorm{v_{H^2}}_{H^1(\Omega)}+ k^2 \bernkopfMelenkNorm{v_{H^2}}_{L^2(\Omega )} \\
      &\leq \underbrace{\bernkopfMelenkNorm{\hat{v}_{H^3}}_{H^2(\Omega)} + k \bernkopfMelenkNorm{\hat{v}_{H^3}}_{H^1(\Omega)} + k^2 \bernkopfMelenkNorm{\hat{v}_{H^3}}_{L^2(\Omega )}}_{\stackrel{\eqref{bernkopf_melenk_eq:v_h_3_hat_low_derivatives}}{\leq} C \bernkopfMelenkNorm{w}_{L^2( \Omega )}} \\
      &\phantom{\leq} + \underbrace{\bernkopfMelenkNorm{z_{H^2}}_{H^2(\Omega)} + k \bernkopfMelenkNorm{z_{H^2}}_{H^1(\Omega)} + k^2 \bernkopfMelenkNorm{z_{H^2}}_{L^2(\Omega )}}_{\stackrel{\eqref{bernkopf_melenk_eq:z_h_2_part}}{\leq} C \bernkopfMelenkNorm{w}_{L^2( \Omega )}} \\
      &\leq C \bernkopfMelenkNorm{w}_{L^2( \Omega )}.
  \end{align*}

	\noindent
  \textbf{Final decomposition of $\pmb{\psi}$}:
  Since $- i k \tilde{\pmb{\psi}}_{H^2} = \nabla ( \tilde{v}_{H^2} - z_{H^2} ) = \nabla \hat{v}_A +  \nabla \hat{v}_{H^3}$,
  we decompose $\tilde{\pmb{\psi}}_{H^2} = \hat{\pmb{\psi}}_A + \hat{\pmb{\psi}}_{H^2}$ accordingly such that $- i k \hat{\pmb{\psi}}_A = \nabla \hat{v}_A$ and consequently $- i k \hat{\pmb{\psi}}_{H^2} = \nabla \hat{v}_{H^3}$.
  The final decomposition takes the form
  \begin{equation*}
    \pmb{\psi} = \tilde{\pmb{\psi}}_A + \tilde{\pmb{\psi}}_{H^2} = \underbrace{\tilde{\pmb{\psi}}_A + \hat{\pmb{\psi}}_A}_{\eqqcolon \pmb{\psi}_A} + \underbrace{\hat{\pmb{\psi}}_{H^2}}_{\eqqcolon \pmb{\psi}_{H^2}}.
  \end{equation*}

	\noindent
  \textbf{Verification of~\eqref{bernkopf_melenk_eq:decomposition_lemma_1}}:
  \begin{align*}
    \bernkopfMelenkNorm{\pmb{\psi}_A}_{H^1(\Omega)} &+ k \bernkopfMelenkNorm{\pmb{\psi}_A}_{L^2(\Omega)} \\
    &\leq \underbrace{\bernkopfMelenkNorm{\tilde{\pmb{\psi}}_A}_{H^1(\Omega)} + k \bernkopfMelenkNorm{\tilde{\pmb{\psi}}_A}_{L^2(\Omega)}}_{\stackrel{\eqref{bernkopf_melenk_eq:psi_A_tilde_low_derivatives}}{\leq} C k \bernkopfMelenkNorm{w}_{L^2( \Omega )}}
        + \bernkopfMelenkNorm{\hat{\pmb{\psi}}_A}_{H^1(\Omega)} + k \bernkopfMelenkNorm{\hat{\pmb{\psi}}_A}_{L^2(\Omega)}  \\
    &\leq C k \bernkopfMelenkNorm{w}_{L^2( \Omega )} + k^{-1}\bernkopfMelenkNorm{\nabla \hat{v}_A}_{H^1(\Omega)} + \bernkopfMelenkNorm{\nabla \hat{v}_A}_{L^2(\Omega)} \\
    &\leq C k \bernkopfMelenkNorm{w}_{L^2( \Omega )}
        + k^{-1}\underbrace{\bernkopfMelenkNorm{\hat{v}_A}_{H^1(\Omega)}}_{\stackrel{\eqref{bernkopf_melenk_eq:v_A_hat_low_derivatives}}{\leq} C \bernkopfMelenkNorm{w}_{L^2( \Omega )}}
        + k^{-1}\underbrace{\bernkopfMelenkNorm{\nabla^2 \hat{v}_A}_{L^2(\Omega)}}_{\stackrel{\eqref{bernkopf_melenk_eq:v_A_hat_high_derivatives}}{\leq} C k \bernkopfMelenkNorm{w}_{L^2( \Omega )}}
        + \underbrace{\bernkopfMelenkNorm{\hat{v}_A}_{H^1(\Omega)}}_{\stackrel{\eqref{bernkopf_melenk_eq:v_A_hat_low_derivatives}}{\leq} C \bernkopfMelenkNorm{w}_{L^2( \Omega )}} \\
    &\leq C k \bernkopfMelenkNorm{w}_{L^2( \Omega )}.
  \end{align*}
	\noindent
  \textbf{Verification of~\eqref{bernkopf_melenk_eq:decomposition_lemma_3}}:
  This is an immediate consequence of~\eqref{bernkopf_melenk_eq:v_A_tilde_high_derivatives},~\eqref{bernkopf_melenk_eq:psi_A_tilde_high_derivatives},~\eqref{bernkopf_melenk_eq:v_A_hat_high_derivatives}, and the fact that $- i k \hat{\pmb{\psi}}_A = \nabla \hat{v}_A$.

	\noindent
  \textbf{Verification of~\eqref{bernkopf_melenk_eq:decomposition_lemma_4}}:
  Since $- i k \hat{\pmb{\psi}}_{H^2} = \nabla \hat{v}_{H^3}$ we estimate
  \begin{align*}
    \bernkopfMelenkNorm{\pmb{\psi}_{H^2}}_{H^2(\Omega)} &+ k \bernkopfMelenkNorm{\pmb{\psi}_{H^2}}_{H^1(\Omega)} + k^2 \bernkopfMelenkNorm{\pmb{\psi}_{H^2}}_{L^2(\Omega)} \\
    &= k^{-1} \bernkopfMelenkNorm{\nabla \hat{v}_{H^3}}_{H^2(\Omega)}
       +      \bernkopfMelenkNorm{\nabla \hat{v}_{H^3}}_{H^1(\Omega)}
       + k    \bernkopfMelenkNorm{\nabla \hat{v}_{H^3}}_{L^2(\Omega)} \\
    &\leq \underbrace{k^{-1} \bernkopfMelenkNorm{\hat{v}_{H^3}}_{H^3(\Omega)}
                      +      \bernkopfMelenkNorm{\hat{v}_{H^3}}_{H^2(\Omega)}
                      + k    \bernkopfMelenkNorm{\hat{v}_{H^3}}_{H^1(\Omega)}}_{\stackrel{\eqref{bernkopf_melenk_eq:v_h_3_hat_low_derivatives}}{\leq} C \bernkopfMelenkNorm{w}_{L^2( \Omega )}}\\
    &\leq C \bernkopfMelenkNorm{w}_{L^2( \Omega )},
  \end{align*}
	which concludes the proof.
  \qed
\end{proof}

\section{Approximation properties of Raviart-Thomas and Brezzi-Douglas-Marini spaces}\label{bernkopf_melenk_section:approximation_propertie_of_raviart_thomas}

In the present Section we analyze the approximation properties of Raviart-Thomas and Brezzi-Douglas-Marini spaces.
To that end, we first state some standard assumptions on the mesh and recall the relevant function spaces.
Next, we will prove the existence of a polynomial approximation operator acting on functions defined on the reference element having certain desirable properties, as outlined below.
This operator will then be used to construct a global polynomial approximation operator by means of the Piola transformation.

\subsection{Preliminaries}

We start with assumptions on the triangulation.
\begin{assumption}[quasi-uniform regular meshes]\label{bernkopf_melenk_assumption:quasi_uniform_regular_meshes}
	Let $\widehat{K}$ be the reference simplex.
  Each element map $F_K \colon \widehat{K} \to K$ can be written as $F_K = R_K \circ A_K$, where $A_K$ is an affine map and the maps $R_K$ and $A_K$ satisfy, for constants $C_\mathrm{affine}, C_\mathrm{metric}, \gamma > 0$ independent of $K$:
  \begin{equation*}
  	\begin{alignedat}{2}
        &\bernkopfMelenkNorm{A^\prime_K}_{L^\infty( \widehat{K} )}        \leq C_\mathrm{affine} h_K, \qquad &&\bernkopfMelenkNorm{ (A^\prime_K)^{-1} }_{L^\infty( \widehat{K} )} \leq C_\mathrm{affine} h^{-1}_K, \\
        &\bernkopfMelenkNorm{ (R^\prime_K)^{-1} }_{L^\infty( \tilde{K} )} \leq C_\mathrm{metric},     \qquad &&\bernkopfMelenkNorm{ \nabla^n R_K }_{L^\infty( \tilde{K} )} \leq C_\mathrm{metric} \gamma^n n! \qquad \forall n \in \mathbb{N}_0.
  	\end{alignedat}
  \end{equation*}
  Here, $\tilde{K} = A_K(\widehat{K})$ and $h_K > 0$ denotes the element diameter.
\end{assumption}

We recall the definition of the Sobolev space $H_{00}^{1/2}(\omega)$. If $\omega$ is an edge of a triangle or face of a tetrahedron, then the norm $\bernkopfMelenkNorm{\cdot}_{H_{00}^{1/2}(\omega)}$ is given by
\begin{equation*}
  \bernkopfMelenkNorm{u}_{H_{00}^{1/2}(\omega)}^2 \coloneqq \bernkopfMelenkNorm{u}_{H^{1/2}(\omega)}^2 + \bernkopfMelenkNorm{\frac{u}{\sqrt{\mathrm{dist}(\cdot, \partial \omega)}}}_{L^2(\omega)}^2,
\end{equation*}
and the space $H_{00}^{1/2}(\omega)$ is the completion of $C_{0}^\infty(\omega)$ under this norm.
Since this norm is induced by a scalar product the space $H_{00}^{1/2}(\omega)$ is a Hilbert space.

On the reference element $\widehat{K}$ we introduce the Raviart-Thomas and Brezzi-Douglas-Marini elements of degree $p \geq 0$ in dimension $d$:
\begin{align*}
  \mathcal{P}_p(\widehat{K}) &\coloneqq \mathrm{span}\left\{ \pmb{x}^{\pmb \alpha} \colon |\pmb{\alpha}| \leq p \right\}, \\
	\pmb{\mathrm{BDM}}_p(\widehat{K}) &\coloneqq \mathcal{P}_p(\widehat{K})^d, \\
  \pmb{\mathrm{RT}}_p(\widehat{K}) &\coloneqq \left\{ \pmb{p} + \pmb{x}q \colon \pmb{p} \in \mathcal{P}_p(\widehat{K})^d, q \in \mathcal{P}_p(\widehat{K}) \right\}.
\end{align*}
Note that trivially $\pmb{\mathrm{BDM}}_p(\widehat{K}) \subset \pmb{\mathrm{RT}}_p(\widehat{K})$.
We also recall the classical Piola transformation, which is the appropriate change of variables for $\pmb{H}(\operatorname{div}, \Omega)$.
For a function $\pmb{\varphi} : K \to \mathbb{R}^d$ and the element map $F_K \colon \widehat{K} \to K$
its Piola transform $\widehat{\pmb{\varphi}} : \widehat{K} \to \mathbb{R}^d$ is given by
\begin{equation*}
	\widehat{\pmb{\varphi}} = (\det F_K^\prime ) (F_K^\prime)^{-1} \pmb{\varphi} \circ F_K.
\end{equation*}
Furthermore we introduce the spaces $S_p(\mathcal{T}_h)$, $\pmb{\mathrm{BDM}}_p(\mathcal{T}_h)$, and $\pmb{\mathrm{RT}}_p(\mathcal{T}_h)$ by standard transformation and (contravariant) Piola transformation respectively:
\begin{align*}
  S_{p}(\mathcal{T}_h) &\coloneqq \left\{ u \in H^1(\Omega) \colon \left.\kern-\nulldelimiterspace{u}\vphantom{\big|} \right|_{K} \circ F_K \in \mathcal{P}_{p}(\widehat{K}) \text{ for all } K \in \mathcal{T}_h \right\}, \\
	\pmb{\mathrm{BDM}}_p(\mathcal{T}_h) &\coloneqq \left\{ \pmb{\varphi} \in \pmb{H}(\operatorname{div}, \Omega) \colon (\det F_K^\prime) (F_K^\prime)^{-1} \left.\kern-\nulldelimiterspace{\pmb{\varphi}}\vphantom{\big|} \right|_{K} \circ F_K \in \pmb{\mathrm{BDM}}_p(\widehat{K}) \text{ for all } K \in \mathcal{T}_h \right\}, \\
  \pmb{\mathrm{RT}}_p(\mathcal{T}_h) &\coloneqq \left\{ \pmb{\varphi} \in \pmb{H}(\operatorname{div}, \Omega) \colon (\det F_K^\prime) (F_K^\prime)^{-1} \left.\kern-\nulldelimiterspace{\pmb{\varphi}}\vphantom{\big|} \right|_{K} \circ F_K \in \pmb{\mathrm{RT}}_p(\widehat{K}) \text{ for all } K \in \mathcal{T}_h \right\}.
\end{align*}

\subsection{Polynomial approximation on the reference element}

We construct a polynomial approximation operator on the reference element $\widehat{K}$:

\begin{definition}\label{bernkopf_melenk_definition:approximation_operator_pi_hat}
  Let $\widehat{K}$ be the reference simplex in $\mathbb{R}^d$, $s > d/2$ and $p \in \mathbb{N}$. We define the operator $\widehat{\Pi}_p: H^s(\widehat{K}) \to \mathcal{P}_p(\widehat{K})$ by the following consecutive minimization steps:
  \begin{enumerate}
    \item Fix $\widehat{\Pi}_pu$ in the vertices: $(\widehat{\Pi}_pu)(\widehat{V}) = u(\widehat{V})$ for all $d+1$ vertices $\widehat{V}$ of $\widehat{K}$. \label{bernkopf_melenk_fix_vertices}
    \item Fix $\widehat{\Pi}_pu$ on the edges: for every edge $\hat{e}$ of $\widehat{K}$ the restriction $\left.\kern-\nulldelimiterspace{(\widehat{\Pi}u_p)}\vphantom{\big|} \right|_{\hat{e}}$ is the unique minimizer of \label{bernkopf_melenk_fix_edges}
    \begin{equation}
      \mathcal{P}_p(\hat{e}) \ni \pi \mapsto p \bernkopfMelenkNorm{u - \pi}_{L^2(\hat{e})}^2 + \bernkopfMelenkNorm{u - \pi}_{H_{00}^{1/2}(\hat{e})}^2, \quad \text{ s.t. }  \pi \text{ satisfies }~\ref{bernkopf_melenk_fix_vertices}.
    \end{equation}
    \item Fix $\widehat{\Pi}_pu$ on the faces (only for $d = 3$): for every face $\hat{f}$ of $\widehat{K}$ the restriction $\left.\kern-\nulldelimiterspace{(\widehat{\Pi}u_p)}\vphantom{\big|} \right|_{\hat{f}}$ is the unique minimizer of \label{bernkopf_melenk_fix_faces}
    \begin{equation}
      \mathcal{P}_p(\hat{f}) \ni \pi \mapsto p^2 \bernkopfMelenkNorm{u - \pi}_{L^2(\hat{f})}^2 + \bernkopfMelenkNorm{u - \pi}_{H^{1}(\hat{f})}^2, \quad \text{ s.t. }  \pi \text{ satisfies }~\ref{bernkopf_melenk_fix_vertices},~\ref{bernkopf_melenk_fix_edges}.
    \end{equation}
    \item Fix $\widehat{\Pi}_pu$ in the volume: $\widehat{\Pi}_pu$ is the unique minimizer of \label{bernkopf_melenk_fix_volume}
    \begin{equation}
      \mathcal{P}_p(\widehat{K}) \ni \pi \mapsto p^2 \bernkopfMelenkNorm{u - \pi}_{L^2(\widehat{K})}^2 + \bernkopfMelenkNorm{u - \pi}_{H^{1}(\widehat{K})}^2, \,\, \text{ s.t. }  \pi \text{ satisfies }~\ref{bernkopf_melenk_fix_vertices},~\ref{bernkopf_melenk_fix_edges},~\ref{bernkopf_melenk_fix_faces}.
    \end{equation}
  \end{enumerate}
\end{definition}

It is convenient to construct an approximant $Iu$ of a function $u$ in an elementwise fashion.
The drawback is that one has to check if the approximant is in fact in the finite element space.
A useful property to achieve this is the following:
The restriction of the approximant $\left.\kern-\nulldelimiterspace{Iu}\vphantom{\big|} \right|_{E}$ to lower dimensional entities $E$ of the mesh, i.e., edges, faces or vertices, is completely determined by the corresponding restriction of $u$.
To put this rigorously, we employ the following concept:

\begin{definition}[restriction property]\label{bernkopf_melenk_definition:restriction_property}
    Let $\widehat{K}$ be the reference simplex in $\mathbb{R}^d$, $s > d/2$, and $p \in \mathbb{N}$.
    A polynomial  $\pi \in \mathcal{P}_p(\widehat{K})$ is said to satisfy the \textit{restriction property} of polynomial degree $p$ for $u \in H^s(\widehat{K})$,
    if it satisfies~\ref{bernkopf_melenk_fix_vertices},~\ref{bernkopf_melenk_fix_edges},~\ref{bernkopf_melenk_fix_faces} of Definition~\ref{bernkopf_melenk_definition:approximation_operator_pi_hat}.
\end{definition}

\begin{remark}\label{bernkopf_melenk_remark:simple_properties_of_pi_hat}
	\smartqed
  Note that minimizations in the definition of the operator $\widehat{\Pi}_p$ are uniquely solvable.
  This is due to the fact these minimizations are constrained minimizations of norms induced by Hilbert spaces.
	These constraints are given by an affine subspace $\mathcal{V}_p^u \leq \mathcal{P}_p(\widehat{K})$, the space of all polynomials satisfying the restriction property for $u$.
  Step~\ref{bernkopf_melenk_fix_volume} is therefore the orthogonal projection onto the space $\mathcal{V}_p^u$ with respect to the scalar product inducing the norm
  \begin{equation*}
    |||u|||^2 \coloneqq p^2 \bernkopfMelenkNorm{u}_{L^2(\widehat{K})}^2 + \bernkopfMelenkNorm{u}_{H^{1}(\widehat{K})}^2.
  \end{equation*}
  Furthermore the  affine space $\mathcal{V}_p^u$ can be written as $\mathcal{V}_p^u = \pi^u + \mathcal{P}_p^0$ for some $\pi^u \in \mathcal{V}_p^u$,
  where $\mathcal{P}_p^0(\widehat{K}) \leq \mathcal{P}_p(\widehat{K})$ is the space of polynomials
vanishing on $\partial \widehat{K}$. The operator $\widehat{\Pi}_p$ can,
apart from being the solution to a minimization problem, also be written as:
  \begin{equation}
\label{bernkopf_melenk_eq:affine-rep-Pi}
    \widehat{\Pi}_p u = \mathrm{argmin}\{ |||u - \pi||| : \pi \in  \mathcal{V}_p^u\} = \pi^u + \widehat{\Pi}_{\mathcal{P}_p^0}(u - \pi^u),
  \end{equation}
  where $\widehat{\Pi}_{\mathcal{P}_p^0}$ denotes the orthogonal projection onto the space $\mathcal{P}_p^0(\widehat{K})$, again with respect to the scalar product inducing $|||\cdot|||$.
  The operator $\widehat{\Pi}_p: H^s(\widehat{K}) \to \mathcal{P}_p(\widehat{K})$ is furthermore linear.
  This is easily seen when one explicitly constructs the Steps~\ref{bernkopf_melenk_fix_vertices},~\ref{bernkopf_melenk_fix_edges},~\ref{bernkopf_melenk_fix_faces} in
Definition~\ref{bernkopf_melenk_definition:approximation_operator_pi_hat}:
  First, one picks polynomials $\pi_{\widehat{V}}$, which are $1$ at the vertex $\widehat{V}$ and zero on all the others.
  Consider the mapping $\widehat{\Pi}_{\widehat{V}} : u \mapsto \sum_{\widehat{V}} u(\widehat{V}) \pi_{\widehat{V}}$.
  This realizes Step~\ref{bernkopf_melenk_fix_vertices}.
  Next one considers the mapping $\tilde{\Pi}_{\hat{e}} : z \mapsto \mathrm{argmin}\{ p \bernkopfMelenkNorm{u - \pi}_{L^2(\hat{e})}^2 + \bernkopfMelenkNorm{u - \pi}_{H_{00}^{1/2}(\hat{e})}^2 : z(\widehat{V}) = 0 \text{ for all vertices } \widehat{V} \}$ and extending it to the reference element.
  Step~\ref{bernkopf_melenk_fix_edges} is then realized by the map $\widehat{\Pi}_{\hat{e}} : u \mapsto  \widehat{\Pi}_{\widehat{V}} u + \tilde{\Pi}_{\hat{e}}( u - \widehat{\Pi}_{\widehat{V}} u )$. One can easily continue this procedure for Step~\ref{bernkopf_melenk_fix_faces} and~\ref{bernkopf_melenk_fix_volume}.
	As a composition of linear operators $\widehat{\Pi}_p$ is therefore also linear.
	\qed
\end{remark}

\begin{remark}
	\smartqed
  Definition~\ref{bernkopf_melenk_definition:restriction_property} of the restriction property was introduced in~\cite[Definition~5.3]{melenk-sauter10} under the name \textit{element-by-element construction}.
  This is due to the fact that, when working in $S_p(\mathcal{T}_h) \leq H^1(\Omega)$, a polynomial, which is constructed in an elementwise fashion on the reference simplex $\widehat{K}$, satisfying the restriction property is already an element of the conforming element space $S_p(\mathcal{T}_h)$.
  However, when working in $\pmb{H}(\operatorname{div}, \Omega)$ or $\pmb{H}(\mathrm{curl}, \Omega)$ one only needs continuity of the inter element normal or tangential trace.
  Furthermore it is necessary to use the Piola transformation to go back and forth between the reference element and the physical element to ensure that normal and tangential vectors are mapped appropriately.
  For the purpose of this paper we therefore use the name restriction property, rather than element-by-element construction.
	\qed
\end{remark}

In the Propositions~\ref{bernkopf_melenk_proposition:theorem_b_4_melenk_sauter_math_comp},~\ref{bernkopf_melenk_proposition:lemma_c_2_melenk_sauter_math_comp}, and~\ref{bernkopf_melenk_proposition:lemma_c_3_melenk_sauter_math_comp} we recall certain useful results concerning approximation properties of polynomials satisfying the restriction property.
These results can be found in~\cite{melenk-sauter10}.

\begin{proposition}[{\cite[Thm.~B.4]{melenk-sauter10}}]\label{bernkopf_melenk_proposition:theorem_b_4_melenk_sauter_math_comp}
  Let $\widehat{K}$ be the reference triangle or reference tetrahedron.
  Let $s>d/2$.
  Then there exists $C>0$ (depending only on $s$ and $d$) and for every $p$ a linear operator $\widehat{\Pi}^{\mathrm{MS}}_p \colon H^s(\widehat{K}) \to \mathcal{P}_p(\widehat{K})$, such that
  $\widehat{\Pi}^{\mathrm{MS}}_p u$ satisfies the restriction property of Definition~\ref{bernkopf_melenk_definition:restriction_property} as well as
  \begin{equation}
    p \bernkopfMelenkNorm{u - \widehat{\Pi}^{\mathrm{MS}}_p u}_{L^2(\widehat{K})} + \bernkopfMelenkNorm{u - \widehat{\Pi}^{\mathrm{MS}}_p u}_{H^1(\widehat{K})} \leq C p^{-(s-1)}|u|_{H^s(\widehat{K})} \qquad \forall p \geq s-1.
  \end{equation}
\end{proposition}

\begin{remark}
	\smartqed
  The operator $\widehat{\Pi}^{\mathrm{MS}}_p$ does in general not preserve polynomials $q \in \mathcal{P}_p(\widehat{K})$.
	See also~\cite{melenk-rojik18} for operators with the projection property.
	\qed
\end{remark}

\begin{proposition}[{\cite[Lemma~C.2]{melenk-sauter10}}]\label{bernkopf_melenk_proposition:lemma_c_2_melenk_sauter_math_comp}
  Let $d \in \{1,2,3\}$, and let $\widehat{K} \subset \mathbb{R}^d$ be the reference simplex.
  Let $\gamma, \tilde{C} > 0$ be given.
  Then there exist constants $C, \sigma > 0$ that depend solely on $\gamma$ and $\tilde{C}$ such that the following is true:
  For any function $u$ that satisfies for some $C_u$, $h$, $R > 0$ and $\kappa>1$ the conditions
  \begin{equation*}
    \bernkopfMelenkNorm{\nabla^n u}_{L^2(\widehat{K})} \leq C_u (\gamma h)^n \max\{n/R, \kappa\}^n \qquad \forall n \in \mathbb{N}_{\geq 2},
  \end{equation*}
  and for any polynomial degree $p \in \mathbb{N}$ that satisfies
  \begin{equation*}
    \frac{h}{R} + \frac{\kappa h}{p} \leq \tilde{C}
  \end{equation*}
  there holds
  \begin{equation*}
    \inf_{\pi \in \mathcal{P}_p(\widehat{K})} \bernkopfMelenkNorm{u-\pi}_{W^{2, \infty}(\widehat{K})} \leq C C_u \left[ \left(\frac{h/R}{\sigma + h/R}\right)^{p+1} + \left(\frac{h \kappa}{\sigma p}\right)^{p+1} \right].
  \end{equation*}
\end{proposition}

\begin{proposition}[{\cite[Lemma~C.3]{melenk-sauter10}}]\label{bernkopf_melenk_proposition:lemma_c_3_melenk_sauter_math_comp}
  Assume the hypotheses of Proposition~\ref{bernkopf_melenk_proposition:lemma_c_2_melenk_sauter_math_comp}.
  Then one can find a polynomial $\pi \in \mathcal{P}_p(\widehat{K})$ that satisfies
  \begin{equation*}
    \bernkopfMelenkNorm{u-\pi}_{W^{1, \infty}(\widehat{K})} \leq C C_u \left[ \left(\frac{h/R}{\sigma + h/R}\right)^{p+1} + \left(\frac{h \kappa}{\sigma p}\right)^{p+1} \right].
  \end{equation*}
  and additionally satisfies the restriction property of Definition~\ref{bernkopf_melenk_definition:restriction_property}.
\end{proposition}

  It is not clear whether the polynomial $\widehat{\Pi}^{\mathrm{MS}}_p u$ has the same approximation properties as the polynomial given by Proposition~\ref{bernkopf_melenk_proposition:lemma_c_3_melenk_sauter_math_comp}.
  However, it is desirable to have both the simultaneous approximation properties in $L^2(\widehat{K})$ and $H^1(\widehat{K})$ as stated in Proposition~\ref{bernkopf_melenk_proposition:theorem_b_4_melenk_sauter_math_comp} as well as
  the exponential approximation properties of an analytic function as stated in Proposition~\ref{bernkopf_melenk_proposition:lemma_c_3_melenk_sauter_math_comp}.
  In the following we will show that the operator $\widehat{\Pi}_p$ constructed in Definition~\ref{bernkopf_melenk_definition:approximation_operator_pi_hat} has these properties.
%

\begin{theorem}[Properties of $\widehat{\Pi}_p$]\label{bernkopf_melenk_theorem:properties_of_pi_hat}
  Let $\widehat{K}$ be the reference triangle or reference tetrahedron.
  Let $s>d/2$.
  Let $\widehat{\Pi}_p \colon H^s(\widehat{K}) \to \mathcal{P}_p(\widehat{K})$ be given by Definition~\ref{bernkopf_melenk_definition:approximation_operator_pi_hat}.
  Then the following holds:
  \begin{enumerate}[(i)]
    \item \label{bernkopf_melenk_eq:properties_of_pi_hat_1} The operator $\widehat{\Pi}_p$ is linear and satisfies the restriction property of Definition~\ref{bernkopf_melenk_definition:restriction_property}.
    \item \label{bernkopf_melenk_eq:properties_of_pi_hat_2} The operator $\widehat{\Pi}_p$ preserves $\mathcal{P}_p(\widehat{K})$, i.e., $\widehat{\Pi}_p q = q$ for all $q \in \mathcal{P}_p(\widehat{K})$.
    \item \label{bernkopf_melenk_eq:properties_of_pi_hat_3} There exists $C_s>0$ (depending only on $s$ and $d$) such that
    \begin{equation*}
      p \bernkopfMelenkNorm{u - \widehat{\Pi}_p u}_{L^2(\widehat{K})} + \bernkopfMelenkNorm{u - \widehat{\Pi}_p u}_{H^1(\widehat{K})} \leq C_s  p^{-(s-1)}|u|_{H^s(\widehat{K})} \qquad \forall p \geq s-1.
    \end{equation*}
    \item \label{bernkopf_melenk_eq:properties_of_pi_hat_4} For given $\gamma$, $\tilde{C} > 0$, there exist constants $C_{A}$, $\sigma > 0$ that depend solely on $\gamma$ and $\tilde{C}$ such that the following is true:
    For any function $u$ and polynomial degree $p$ that satisfy the assumptions of Proposition~\ref{bernkopf_melenk_proposition:lemma_c_2_melenk_sauter_math_comp}
    there holds
    \begin{equation*}
      \bernkopfMelenkNorm{u - \widehat{\Pi}_p u}_{W^{1, \infty}(\widehat{K})} \leq C_A C_u \left[ \left(\frac{h/R}{\sigma + h/R}\right)^{p+1} + \left(\frac{h \kappa}{\sigma p}\right)^{p+1} \right].
    \end{equation*}
  \end{enumerate}
\end{theorem}

\noindent
\textbf{Idea}:
The crucial points of Theorem~\ref{bernkopf_melenk_theorem:properties_of_pi_hat} are items~\eqref{bernkopf_melenk_eq:properties_of_pi_hat_3} and~\eqref{bernkopf_melenk_eq:properties_of_pi_hat_4}.
To verify~\eqref{bernkopf_melenk_eq:properties_of_pi_hat_3} we will exploit the approximation properties of $\widehat{\Pi}^{\mathrm{MS}}_p$ given by Proposition~\ref{bernkopf_melenk_proposition:theorem_b_4_melenk_sauter_math_comp}
together with the fact that $\widehat{\Pi}_pu$ is the solution to a minimization problem.
To prove~\eqref{bernkopf_melenk_eq:properties_of_pi_hat_4} we use the affine projection representation~\eqref{bernkopf_melenk_eq:affine-rep-Pi} of $\widehat{\Pi}_p$
together with the approximation properties of polynomials satisfying the restriction property given in
Proposition~\ref{bernkopf_melenk_proposition:lemma_c_3_melenk_sauter_math_comp}.

\begin{proof}
	\smartqed
  Assertion~\eqref{bernkopf_melenk_eq:properties_of_pi_hat_1} is trivially satisfied due to the construction in Definition~\ref{bernkopf_melenk_definition:approximation_operator_pi_hat} and Remark~\ref{bernkopf_melenk_remark:simple_properties_of_pi_hat}.
  \\
  \noindent
  Assertion~\eqref{bernkopf_melenk_eq:properties_of_pi_hat_2} is also trivially satisfied, since for a given polynomial $q \in \mathcal{P}_p(\widehat{K})$ the norms in Definition~\ref{bernkopf_melenk_definition:approximation_operator_pi_hat} are minimized at $q$.
  \\
  \noindent
  To prove Assertion~\eqref{bernkopf_melenk_eq:properties_of_pi_hat_3} recall that Step~\ref{bernkopf_melenk_fix_volume} in Definition~\ref{bernkopf_melenk_definition:approximation_operator_pi_hat} is exactly the minimization of the norm in question, constrained to all polynomials satisfying the restriction property for $u$.
  Since $\widehat{\Pi}^{\mathrm{MS}}_p u$ given by Proposition~\ref{bernkopf_melenk_proposition:theorem_b_4_melenk_sauter_math_comp} also satisfies the restriction property we can immediately conclude for $p \ge s-1$ that
  \begin{align*}
    p \bernkopfMelenkNorm{u - \widehat{\Pi}_p u}_{L^2(\widehat{K})} + \bernkopfMelenkNorm{u - \widehat{\Pi}_p u}_{H^1(\widehat{K})}
    &\leq p \bernkopfMelenkNorm{u - \widehat{\Pi}^{\mathrm{MS}}_p u}_{L^2(\widehat{K})} + \bernkopfMelenkNorm{u - \widehat{\Pi}^{\mathrm{MS}}_p u}_{H^1(\widehat{K})} \\
    & \leq C_s  p^{-(s-1)}|u|_{H^s(\widehat{K})}.
  \end{align*}
  \\
  \noindent
  We turn to Assertion~\eqref{bernkopf_melenk_eq:properties_of_pi_hat_4}.
  Since polynomials up to degree $p$ are preserved under $\widehat{\Pi}_p$, we immediately have
  \begin{equation}\label{bernkopf_melenk_eq:foo-1}
    \bernkopfMelenkNorm{u - \widehat{\Pi}_p u}_{W^{1, \infty}(\widehat{K})} \leq \bernkopfMelenkNorm{u - q}_{W^{1, \infty}(\widehat{K})} + \bernkopfMelenkNorm{\widehat{\Pi}_p q - \widehat{\Pi}_p u}_{W^{1, \infty}(\widehat{K})},
  \end{equation}
  for any $q \in \mathcal{P}_p(\widehat{K})$.
  We estimate the second term in~\eqref{bernkopf_melenk_eq:foo-1}.
  We have seen in~\eqref{bernkopf_melenk_eq:affine-rep-Pi} that the operator $\widehat{\Pi}_p$ can be written as $\widehat{\Pi}_p u = \pi^u + \widehat{\Pi}_{\mathcal{P}_p^0}(u - \pi^u)$ for any $\pi^u \in \mathcal{V}_p^u$ (the affine space of polynomials with restriction property for $u$), where $\widehat{\Pi}_{\mathcal{P}_p^0}$ is the orthogonal projection onto $\mathcal{P}_p^0(\widehat{K}) \leq \mathcal{P}_p(\widehat{K})$, the space of polynomials vanishing on $\partial \widehat{K}$,
with respect to the norm $|||\cdot|||$.
  Therefore we have
  \begin{equation*}
    \widehat{\Pi}_p q - \widehat{\Pi}_p u = \pi^q - \pi^u + \widehat{\Pi}_{\mathcal{P}_p^0}(q - u + \pi^u - \pi^q)
  \end{equation*}
  for any $\pi^u \in \mathcal{V}_p^u$ and $\pi^q \in \mathcal{V}_p^q$.
  Selecting $q \in \mathcal{V}_p^u$ allows us to choose $\pi^u = \pi^q = q$, which immediately gives
  \begin{equation*}
    \widehat{\Pi}_p q - \widehat{\Pi}_p u = \widehat{\Pi}_{\mathcal{P}_p^0}(q - u)
  \end{equation*}
  for all $q \in \mathcal{V}_p^u$.
  Using the polynomial inverse estimates
$\bernkopfMelenkNorm{\pi}_{L^\infty(\Omega)} \leq C p^d \bernkopfMelenkNorm{\pi}_{L^2(\Omega)}$ for all $\pi \in \mathcal{P}_p(\widehat{K})$,
(see, e.g.,~\cite[Thm.~4.76]{schwab98} for the case $d = 2$),
we find
  \begin{equation*}
    \bernkopfMelenkNorm{\widehat{\Pi}_p q - \widehat{\Pi}_p u}_{W^{1, \infty}(\widehat{K})} = \bernkopfMelenkNorm{\widehat{\Pi}_{\mathcal{P}_p^0}(q - u)}_{W^{1, \infty}(\widehat{K})} \lesssim p^d \bernkopfMelenkNorm{\widehat{\Pi}_{\mathcal{P}_p^0}(q - u)}_{H^1(\widehat{K})}.
  \end{equation*}
  Since $\widehat{\Pi}_{\mathcal{P}_p^0}$ is the orthogonal projection with respect to the norm $|||\cdot|||$
we obtain
  \begin{equation*}
    p^d \bernkopfMelenkNorm{\widehat{\Pi}_{\mathcal{P}_p^0}(q - u)}_{H^1(\widehat{K})} \leq p^d |||q - u|||
\lesssim p^{d+1} \bernkopfMelenkNorm{q - u}_{W^{1,\infty}(\widehat{K})}.
  \end{equation*}
  We therefore conclude that
  \begin{equation*}
    \bernkopfMelenkNorm{u - \widehat{\Pi}_p u}_{W^{1, \infty}(\widehat{K})} \lesssim p^{d+1} \bernkopfMelenkNorm{u - q}_{W^{1, \infty}(\widehat{K})}
  \end{equation*}
  for all $q \in \mathcal{V}^u_p$.
  Proposition~\ref{bernkopf_melenk_proposition:lemma_c_3_melenk_sauter_math_comp} provides a polynomial $q \in \mathcal{V}_p^u$
  with the desired approximation properties.  Absorbing the algebraic factor $p^{d+1}$ into the exponential factor
  then yields the result.
  \qed
\end{proof}

\subsection{$\pmb{H}(\operatorname{div}, \Omega)$-conforming approximation operators}

In the following we will construct an approximation operator
$\pmb{\Pi}^{\operatorname{div},s}_p \colon \pmb{H}^s(\Omega) \to \pmb{\mathrm{BDM}}_p(\mathcal{T}_h) \subset \pmb{\mathrm{RT}}_p(\mathcal{T}_h)$ that features
the optimal convergence rates in $p$ simultaneously in $L^2(\Omega)$ and $\pmb{H}(\operatorname{div},\Omega)$ for $s > d/2$.
The operator will act elementwise.
First we consider any operator $\widehat{\pmb{\Pi}}^{\operatorname{div},s}_p \colon \pmb{H}^s(\widehat{K}) \to \pmb{\mathrm{BDM}}_p(\widehat{K}) \subset  \pmb{\mathrm{RT}}_p(\widehat{K})$
and define $\pmb{\Pi}^{\operatorname{div},s}_p$ on $\pmb{H}^s(\Omega)$ elementwise using the Piola transformation by
\begin{equation}
  \left.\kern-\nulldelimiterspace{\left(\pmb{\Pi}^{{\operatorname{div}},s}_p \pmb{\varphi}\right)}\vphantom{\big|} \right|_{K} \coloneqq
  \left[(\det F_K^\prime )^{-1} F_K^\prime \widehat{\pmb{\Pi}}^{\operatorname{div}, s}_p \left[ (\det F_K^\prime ) (F_K^\prime)^{-1} \pmb{\varphi} \circ F_K \right] \right] \circ F_K^{-1}.
\end{equation}
In order for $\pmb{\Pi}^{\operatorname{div},s}_p$ to map into the conforming finite element space
one has to select the operator $\widehat{\pmb{\Pi}}^{\operatorname{div}, s}_p$ correctly.
We choose $\widehat{\pmb{\Pi}}^{\operatorname{div}, s}_p \colon $ $ \pmb{H}^s(\widehat{K}) \to \mathcal{P}_p(\widehat{K})^d = \pmb{\mathrm{BDM}}_p(\widehat{K}) \subset \pmb{\mathrm{RT}}_p(\widehat{K})$
to be the componentwise application of $\widehat{\Pi}_p$ from Definition~\ref{bernkopf_melenk_definition:approximation_operator_pi_hat} and analyzed in Theorem~\ref{bernkopf_melenk_theorem:properties_of_pi_hat}:
\begin{equation}
  \left(\widehat{\pmb{\Pi}}^{\operatorname{div}, s}_p \pmb{\varphi}\right)_i \coloneqq \widehat{\Pi}_p \pmb{\varphi}_i, \qquad \text{for } i = 1, \dots, d.
\end{equation}
This choice will ensure the desired approximation properties, and will also map into the conforming finite element space due to the restriction property.
We will summarize and prove certain properties of the above constructed operators $\widehat{\pmb{\Pi}}^{\operatorname{div}, s}_p$ and $\pmb{\Pi}^{\operatorname{div}, s}_p$.
See \cite{melenk-sauter18} for a similar construction concerning the space $\pmb{H}(\mathrm{curl}, \Omega)$.
\begin{lemma}\label{bernkopf_melenk_lemma:reference_element_RT_interpolation_operator}
  Let $s > d/2$ and let the operators $\widehat{\pmb{\Pi}}^{\operatorname{div}, s}_p$ and $\pmb{\Pi}^{\operatorname{div}, s}_p$ be defined as above.
  Then there holds:
  \begin{enumerate}[(i)]
    \item The operator $\widehat{\pmb{\Pi}}^{\operatorname{div}, s}_p \colon  \pmb{H}^s(\widehat{K}) \to \pmb{\mathrm{BDM}}_p(\widehat{K}) \subset \pmb{\mathrm{RT}}_p(\widehat{K})$ satisfies for $p \geq s-1$
    \begin{equation}
      p \bernkopfMelenkNorm{\pmb{\varphi} - \widehat{\pmb{\Pi}}^{\operatorname{div}, s}_p \pmb{\varphi}}_{L^2(\widehat{K})} + \bernkopfMelenkNorm{\pmb{\varphi} - \widehat{\pmb{\Pi}}^{\operatorname{div}, s}_p \pmb{\varphi}}_{H^1(\widehat{K})} \lesssim p^{-(s-1)}|\pmb{\varphi}|_{H^s(\widehat{K})}.
    \end{equation}
    \item Under the assumptions Theorem~\ref{bernkopf_melenk_theorem:properties_of_pi_hat},~\eqref{bernkopf_melenk_eq:properties_of_pi_hat_4}
there holds for some constants $C_A$, $\sigma> 0$ independent of $p$, $h$, $R$
    \begin{equation*}
      \bernkopfMelenkNorm{\pmb{\varphi} - \widehat{\pmb{\Pi}}^{\operatorname{div}, s}_p \pmb{\varphi}}_{W^{1, \infty}(\widehat{K})} \leq C_A C_{\pmb{\varphi}} \left[ \left(\frac{h/R}{\sigma + h/R}\right)^{p+1} + \left(\frac{h \kappa}{\sigma p}\right)^{p+1} \right].
    \end{equation*}
    \item The operator $\pmb{\Pi}^{\operatorname{div}, s}_p$ defined on $\pmb{H}^s(\Omega)$ maps to the conforming space $\pmb{\mathrm{BDM}}_p(\mathcal{T}_h) \subset \pmb{\mathrm{RT}}_p(\mathcal{T}_h)$.
  \end{enumerate}
\end{lemma}

\begin{proof}
	\smartqed
  The first two assertions hold by construction and
  Theorem~\ref{bernkopf_melenk_theorem:properties_of_pi_hat}, properties~\eqref{bernkopf_melenk_eq:properties_of_pi_hat_3},~\eqref{bernkopf_melenk_eq:properties_of_pi_hat_4}.
  To prove the third assertion, note that $\widehat{\pmb{\Pi}}^{\operatorname{div}, s}_p$ maps to $\pmb{\mathrm{BDM}}_p(\widehat{K})$ so that
  \begin{equation}
    (\det F_K^\prime) (F_K^\prime)^{-1}   \left.\kern-\nulldelimiterspace{\left(\pmb{\Pi}^{{\operatorname{div}},s}_p \pmb{\varphi}\right)}\vphantom{\big|} \right|_{K} \circ F_K \in \pmb{\mathrm{BDM}}_p(\widehat{K}) \qquad \text{ for all } K \in \mathcal{T}_h,
  \end{equation}
  by construction.
  We are therefore left with verifying that $\pmb{\Pi}^{{\operatorname{div}},s}_p \pmb{\varphi} \in \pmb{H}(\operatorname{div}, \Omega)$.
  Since $\pmb{\Pi}^{{\operatorname{div}},s}_p \pmb{\varphi}$ is piecewise smooth it suffices to show inter element continuity of the normal trace.
  We will first show that the normal trace of $\widehat{\pmb{\Pi}}^{\operatorname{div}, s}_p \pmb{\varphi}$ in fact only depends on the normal trace of $\pmb{\varphi}$.
  Consider a face $\hat{f}$ of $\widehat{K}$.
	Let $\gamma_{\hat{\pmb{n}}_{\hat{f}}}$ denote the normal trace for the face $\hat{f}$.
  We calculate
  \begin{align*}
    \gamma_{\hat{\pmb{n}}_{\hat{f}}}\left( \widehat{\pmb{\Pi}}^{\operatorname{div}, s}_p \pmb{\varphi} \right)
    &= \left.\kern-\nulldelimiterspace{\left(\widehat{\pmb{\Pi}}^{\operatorname{div}, s}_p \pmb{\varphi} \right)}\vphantom{\big|} \right|_{\hat{f}} \cdot \hat{\pmb{n}}_{\hat{f}}
    = \left.\kern-\nulldelimiterspace{ \begin{pmatrix} \widehat{\Pi}_p \pmb{\varphi}_1  \\ \vdots \\ \widehat{\Pi}_p \pmb{\varphi}_d \end{pmatrix} }\vphantom{\big|} \right|_{\hat{f}} \cdot \hat{\pmb{n}}_{\hat{f}} \\
    &= \begin{pmatrix} \widehat{\Pi}_p ( \left.\kern-\nulldelimiterspace{\pmb{\varphi}_1}\vphantom{\big|} \right|_{\hat{f}} )  \\ \vdots \\ \widehat{\Pi}_p ( \left.\kern-\nulldelimiterspace{\pmb{\varphi}_d}\vphantom{\big|} \right|_{\hat{f}} ) \end{pmatrix} \cdot \hat{\pmb{n}}_{\hat{f}}
    = \widehat{\Pi}_p ( \pmb{\varphi} \cdot \hat{\pmb{n}}_{\hat{f}} )
    = \widehat{\Pi}_p ( \gamma_{\hat{\pmb{n}}_{\hat{f}}}{\pmb{\varphi}} ).
  \end{align*}
  Here we used that the operator $\widehat{\Pi}_p$ satisfies the restriction property and the fact that $\hat{\pmb{n}}_{\hat{f}}$ is constant on $\hat{f}$.
  Furthermore note that we abused notation in that the symbol $\widehat{\Pi}_p$ is used both for the $d$ dimensional as well as the $d-1$ dimensional version.
  We conclude the proof using the fact that if $\hat{\pmb{n}}$ is the unit outward normal to $\widehat{K}$ the vector $\pmb{n}$ on $K$ given by
  \begin{equation*}
    \pmb{n} \circ F_K = \frac{1}{ \lVert (F_K^\prime)^{-T} \hat{\pmb{n}} \rVert} (F_K^\prime)^{-T} \hat{\pmb{n}}
  \end{equation*}
  is a unit normal to $K$, see, e.g.,~\cite[Section 3.9 and 5.4]{monk03}.
  \qed
\end{proof}
We have $p$-optimal approximation properties on the reference element $\widehat{K}$ by the operator $\widehat{\pmb{\Pi}}^{\operatorname{div}, s}_p$.

%

\begin{corollary}[Approximation of $H^s(\Omega)$ functions]\label{bernkopf_melenk_corollary:global_RT_interpolation_operator}
  For $d = 2,3$ and $s > d/2$ the operator $\pmb{\Pi}^{\operatorname{div}, s}_p \colon \pmb{H}^s(\Omega) \to \pmb{\mathrm{BDM}}_p(\mathcal{T}_h) \subset \pmb{\mathrm{RT}}_p(\mathcal{T}_h)$ satisfies
  \begin{equation*}
    \frac{p}{h} \bernkopfMelenkNorm{\pmb{\varphi} - \pmb{\Pi}^{\operatorname{div}, s}_p \pmb{\varphi}}_{L^2(\Omega)} + \bernkopfMelenkNorm{\pmb{\varphi} - \pmb{\Pi}^{\operatorname{div}, s}_p \pmb{\varphi}}_{H^1(\mathcal{T}_h)} \lesssim \left(\frac{h}{p}\right)^{s-1} \bernkopfMelenkNorm{\pmb{\varphi}}_{H^s(\Omega)} \quad \forall p \ge s-1,
  \end{equation*}
	where $\bernkopfMelenkNorm{\cdot}_{H^1(\mathcal{T}_h)}$ denotes the broken $H^1$-norm.
\end{corollary}

\begin{proof}
	\smartqed
	The proof follows from Lemma~\ref{bernkopf_melenk_lemma:reference_element_RT_interpolation_operator} together with a scaling argument.
  \qed
\end{proof}

\begin{corollary}[Approximation of analytic functions]\label{bernkopf_melenk_corollary:global_RT_interpolation_operator_analytic}
  Let $\pmb{\varphi}$ satisfy, for some $C_{\pmb{\varphi}}$, $\gamma>0$,
  \begin{equation*}
    \bernkopfMelenkNorm{\nabla^n \pmb{\varphi}}_{L^2(\Omega)} \leq C_{\pmb{\varphi}} \gamma^n \max(n, k)^n \qquad \forall n \in \mathbb{N}_{0}.
  \end{equation*}
There exist $C$, $\sigma > 0$ independent of $h$, $p$, and $k$ such that
  \begin{align*}
    \bernkopfMelenkNorm{\pmb{\varphi} - \pmb{\Pi}^{\operatorname{div}, s}_p \pmb{\varphi}}_{H^1(\mathcal{T}_h)} &+ k \bernkopfMelenkNorm{\pmb{\varphi} - \pmb{\Pi}^{\operatorname{div}, s}_p \pmb{\varphi}}_{L^2(\Omega)} \\
    &\leq C C_{\pmb{\varphi}} \left[ \left(\frac{h}{h + p}\right)^p \left(1 + \frac{h k}{h + \sigma}\right) + k \left(\frac{k h}{\sigma p}\right)^p \left(\frac{1}{p} + \frac{k h}{\sigma p}\right) \right].
  \end{align*}
\end{corollary}

\begin{proof}
	\smartqed
  We mimic the procedure of~\cite[Thm.~5.5]{melenk-sauter10} and~\cite[Lemma~4.7]{chen-qiu17}.
  First consider for each element $K \in \mathcal{T}_h$ the constant $C_K$ given by
  \begin{equation*}
    C_K^2 \coloneqq \sum_{n \geq 0} \frac{\bernkopfMelenkNorm{\nabla^n \pmb{\varphi}}_{L^2(K)}^2}{(2 \gamma \max(n, k))^{2n}},
  \end{equation*}
  which is finite by assumption.
  Note that we immediately have
  \begin{align*}
    \bernkopfMelenkNorm{\nabla^n \pmb{\varphi}}_{L^2(K)} &\leq 2^n \gamma^n \max(n,k)^n C_K, \\
    \sum_{K \in \mathcal{T}_h} C_K^2 &\leq \frac{4}{3} C_{\pmb{\varphi}}^2.
  \end{align*}
  We write $\widehat{\pmb{\varphi}}$ as
  \begin{align*}
    \widehat{\pmb{\varphi}}
    &= \det(F_K^\prime) (F_K^\prime)^{-1} \pmb{\varphi} \circ F_K
    = \det(R_K^\prime \circ A_K A_K^\prime) (R_K^\prime \circ A_K A_K^\prime)^{-1} \pmb{\varphi} \circ F_K \\
    & = \det (A_K^\prime) (A_K^\prime)^{-1} \tilde{\pmb{\varphi}} \circ A_K,
  \end{align*}
  with
  \begin{equation*}
    \tilde{\pmb{\varphi}} = \det(R_K^\prime) (R_K^\prime)^{-1} \pmb{\varphi} \circ R_K.
  \end{equation*}
  As in~\cite[Lemma~C.1]{melenk-sauter10} for simple changes of variables,
  we apply~\cite[Lemma~4.3.1]{melenk02} to the function $\tilde{\pmb{\varphi}}$ and obtain
  the existence of constants $\overline{\gamma}$, $C > 0$ depending additionally on the constants
  describing the analyticity of the map $R_K$ such that
  \begin{equation*}
    \bernkopfMelenkNorm{\nabla^n \tilde{\pmb{\varphi}}}_{L^2(\tilde{K})} \leq C \overline{\gamma}^n \max(n,k)^n C_K \qquad \forall n \in \mathbb{N}_{0}.
  \end{equation*}
  Since $A_K$ is affine we immediately deduce that
  \begin{equation*}
    \bernkopfMelenkNorm{\nabla^n \widehat{\pmb{\varphi}}}_{L^2(\widehat{K})}
		\lesssim h^{d/2-1} h^n \bernkopfMelenkNorm{\nabla^n \tilde{\pmb{\varphi}}}_{L^2(\tilde{K})}
		\leq h^{d/2-1}  (\overline{\gamma} h)^n \max(n,k)^n C_K \quad \forall n \in \mathbb{N}_{n \geq 1}.
  \end{equation*}
  Hence by Lemma~\ref{bernkopf_melenk_lemma:reference_element_RT_interpolation_operator} with $R = 1$ we have
  \begin{equation*}
    \bernkopfMelenkNorm{\widehat{\pmb{\varphi}} - \widehat{\pmb{\Pi}}^{\operatorname{div}, s}_p \widehat{\pmb{\varphi}}}_{W^{1, \infty}(\widehat{K})} \lesssim C_K h^{d/2-1} \left[ \left(\frac{h}{\sigma + h}\right)^{p+1} + \left(\frac{h k}{\sigma p}\right)^{p+1} \right]
  \end{equation*}
	for some $\sigma > 0$.
  By a change of variables there holds for $q = 0$, $1$
  \begin{align*}
    \bernkopfMelenkNorm{\pmb{\varphi} - \pmb{\Pi}^{\operatorname{div}, s}_p \pmb{\varphi}}_{H^q(K)} &\lesssim h^{-d/2 + 1 - q} \bernkopfMelenkNorm{\widehat{\pmb{\varphi}} - \widehat{\pmb{\Pi}}^{\operatorname{div}, s}_p \widehat{\pmb{\varphi}}}_{H^q(\widehat{K})} \\
    & \lesssim h^{-q} C_K \left[ \left(\frac{h}{\sigma + h}\right)^{p+1} + \left(\frac{h k}{\sigma p}\right)^{p+1} \right].
  \end{align*}
  Summation over all elements gives
  \begin{align*}
    & \bernkopfMelenkNorm{\pmb{\varphi} - \pmb{\Pi}^{\operatorname{div}, s}_p \pmb{\varphi}}_{H^1(\mathcal{T}_h)} + k \bernkopfMelenkNorm{\pmb{\varphi} - \pmb{\Pi}^{\operatorname{div}, s}_p \pmb{\varphi}}_{L^2(\Omega)} \\
    &\lesssim \left[ \left(\frac{h}{\sigma + h}\right)^{p} + k \left(\frac{h}{\sigma + h}\right)^{p+1} + \frac{k}{p} \left(\frac{h k}{\sigma p}\right)^{p} + k \left(\frac{h k}{\sigma p}\right)^{p+1} \right] \sqrt{\sum_{K \in \mathcal{T}_h} C_K^2} \\
    &\lesssim \left[ \left(\frac{h}{h + p}\right)^p \left(1 + \frac{h k}{h + \sigma}\right) + k \left(\frac{k h}{\sigma p}\right)^p \left(\frac{1}{p} + \frac{k h}{\sigma p}\right) \right] C_{\pmb{\varphi}},
  \end{align*}
  which completes the proof.
  \qed
\end{proof}

\section{A priori estimate}\label{bernkopf_melenk_section:a_priori_estimate}

We now turn to an {\sl a priori} estimate of the FOSLS method.
Again the proof follows the ideas of~\cite[Lemma~5.1]{chen-qiu17},
resting, however, on the refined duality argument given in Lemma~\ref{bernkopf_melenk_lemma:duality_argument}
and the approximation properties derived in Section~\ref{bernkopf_melenk_section:approximation_propertie_of_raviart_thomas}
to obtain the factor $h/p$. For the readers' convenience we recapitulate the important steps.
As in~\cite{melenk-sauter10} we show that this can be achieved under the conditions $kh/p$ sufficiently small
and $p$ of order $\log k$.

\begin{theorem}[A priori estimate]\label{bernkopf_melenk_theorem:a_priori_estimate}
Let Assumptions~\ref{bernkopf_melenk_assumption:general_assumptions},~\ref{bernkopf_melenk_assumption:quasi_uniform_regular_meshes} be valid.
Then there exist constants $c_1$, $c_2 > 0$ that are independent of $h$, $p$, and $k$ such that the conditions
  \begin{equation}\label{bernkopf_melenk_eq:scale_resolution_condition}
    \frac{kh}{p} \leq c_1 \qquad \text{and} \qquad p \geq c_2(\log k + 1)
  \end{equation}
  imply that the approximation $(\pmb{\varphi}_h, u_h)$ of the FOSLS method satisfies the following:
  For any $(\pmb{\psi}_h, v_h) \in \pmb{V}_h \times W_h$ there holds
  \begin{align*}
    \bernkopfMelenkNorm{u-u_h}_{L^2(\Omega)}
		&\lesssim
		\frac{h}{p}
    \Big(  \bernkopfMelenkNorm{ \nabla ( u - v_h )}_{L^2(\Omega)} +
		 				k \bernkopfMelenkNorm{ u - v_h }_{L^2(\Omega)} + \\
					 &\bernkopfMelenkNorm{ \nabla \cdot ( \pmb{\varphi} - \pmb{\psi}_h )}_{L^2(\Omega) } +
            k \bernkopfMelenkNorm{ \pmb{\varphi} - \pmb{\psi}_h }_{L^2(\Omega) } +
            k^{1/2} \bernkopfMelenkNorm{ ( \pmb{\varphi} - \pmb{\psi}_h ) \cdot \pmb{n} }_{L^2(\partial \Omega)}
		\Big).
  \end{align*}
\end{theorem}

\begin{proof}
	\smartqed
  Let $e^u = u - u_h$ and $\pmb{e^\varphi} = \pmb{\varphi} - \pmb{\varphi}_h$ denote the errors of the two components.
  We apply the duality argument from Lemma~\ref{bernkopf_melenk_lemma:duality_argument} with $w = e^u$ and also apply the corresponding splitting:
  \begin{equation*}
    \bernkopfMelenkNorm{e^u}_{L^2(\Omega)}^2 =
    b( (\pmb{e^\varphi}, e^u), (\pmb{\psi}, v) ) =
    b( (\pmb{e^\varphi}, e^u), (\pmb{\psi}_A, v_A) ) + b( (\pmb{e^\varphi}, e^u), (\pmb{\psi}_{H^2}, v_{H^2}) ).
  \end{equation*}
  Exploiting the Galerkin orthogonality we have
  \begin{equation*}
    \bernkopfMelenkNorm{e^u}_{L^2(\Omega)}^2 =
    b( (\pmb{e^\varphi}, e^u), (\pmb{\psi}_A - \tilde{\pmb{\psi}}_A, v_A - \tilde{v}_A) ) + b( (\pmb{e^\varphi}, e^u), ( \pmb{\psi}_{H^2} - \tilde{\pmb{\psi}}_{H^2}, v_{H^2} - \tilde{v}_{H^2}) ).
  \end{equation*}
  for any $(\tilde{\pmb{\psi}}_A, \tilde{v}_A), (\tilde{\pmb{\psi}}_{H^2}, \tilde{v}_{H^2}) \in \pmb{V}_h \times W_h$.
  Using Cauchy-Schwarz we arrive at
  \begin{align}\label{bernkopf_melenk_eq:first_cs_estimate_in_a_priori_estimate}
		\begin{split}
    & \bernkopfMelenkNorm{e^u}_{L^2(\Omega)}^2 \lesssim
		\left[ \bernkopfMelenkNorm{i k \pmb{e^\varphi} + \nabla e^u}_{L^2(\Omega)} + \bernkopfMelenkNorm{i k e^u + \nabla \cdot \pmb{e^\varphi}}_{L^2(\Omega)} + k^{1/2} \bernkopfMelenkNorm{ \pmb{e^\varphi} \cdot \pmb{n} + e^u}_{L^2(\partial\Omega)} \right] \cdot \\
		& \Big( \bernkopfMelenkNorm{\nabla \cdot ( \pmb{\psi}_A - \tilde{\pmb{\psi}}_A )}_{L^2(\Omega)} +
			k \bernkopfMelenkNorm{ \pmb{\psi}_A - \tilde{\pmb{\psi}}_A }_{L^2(\Omega)} +
			k^{1/2} \bernkopfMelenkNorm{ ( \pmb{\psi}_A - \tilde{\pmb{\psi}}_A ) \cdot \pmb{n} }_{L^2(\partial \Omega)} + \\
		& \phantom{\Big(} \bernkopfMelenkNorm{\nabla \cdot ( \pmb{\psi}_{H^2} - \tilde{\pmb{\psi}}_{H^2} )}_{L^2(\Omega)} +
			k \bernkopfMelenkNorm{\pmb{\psi}_{H^2} - \tilde{\pmb{\psi}}_{H^2}}_{L^2(\Omega)} +
			k^{1/2} \bernkopfMelenkNorm{ ( \pmb{\psi}_{H^2} - \tilde{\pmb{\psi}}_{H^2} ) \cdot \pmb{n} }_{L^2(\partial \Omega)} + \\
		& \phantom{\Big(} \bernkopfMelenkNorm{\nabla ( v_A - \tilde{v}_A )}_{L^2(\Omega)} +
			k \bernkopfMelenkNorm{ v_A - \tilde{v}_A }_{L^2(\Omega)} +
			k^{1/2} \bernkopfMelenkNorm{ v_A - \tilde{v}_A }_{L^2(\partial \Omega)} + \\
		& \phantom{\Big(} \bernkopfMelenkNorm{\nabla ( v_{H^2} - \tilde{v}_{H^2} )}_{L^2(\Omega)} +
			k \bernkopfMelenkNorm{v_{H^2} - \tilde{v}_{H^2}}_{L^2(\Omega)} +
			k^{1/2} \bernkopfMelenkNorm{v_{H^2} - \tilde{v}_{H^2}}_{L^2(\partial \Omega)} \Big).
		\end{split}
  \end{align}
  We are going to exploit the approximation properties in the corresponding norms and spaces.

	\noindent
  \textbf{Approximation of $v_A$ and $v_{H^2}$}:
  For the approximation we may apply~\cite[Lemma~4.10]{chen-qiu17}, which is essentially the procedure of~\cite[Thm.~5.5]{melenk-sauter10} together with a multiplicative trace inequality.
  Using the estimates~\eqref{bernkopf_melenk_eq:decomposition_lemma_2},~\eqref{bernkopf_melenk_eq:decomposition_lemma_3}, and~\eqref{bernkopf_melenk_eq:decomposition_lemma_5} in Lemma~\ref{bernkopf_melenk_lemma:duality_argument} as well as~\cite[Thm.~B.4]{melenk-sauter10} to find appropriate approximations $\tilde{v}_{H^2}$ and $\tilde{v}_A$ we have
  \begin{align*}
    & \bernkopfMelenkNorm{\nabla ( v_A - \tilde{v}_A )}_{L^2(\Omega)} + k \bernkopfMelenkNorm{ v_A - \tilde{v}_A }_{L^2(\Omega)} + k^{1/2} \bernkopfMelenkNorm{ v_A - \tilde{v}_A }_{L^2(\partial \Omega)} \\
    & \qquad \lesssim \left[ \left(\frac{h}{h + p}\right)^p \left(1 + \frac{h k}{h + \sigma}\right) + k \left(\frac{k h}{\sigma p}\right)^p \left(\frac{1}{p} + \frac{k h}{\sigma p}\right) \right] \bernkopfMelenkNorm{e^u}_{L^2(\Omega)} \\
    & \qquad \lesssim \frac{h}{p} \bernkopfMelenkNorm{e^u}_{L^2(\Omega)}
  \end{align*}
  as well as
  \begin{align*}
    & \bernkopfMelenkNorm{\nabla ( v_{H^2} - \tilde{v}_{H^2} )}_{L^2(\Omega)} + k \bernkopfMelenkNorm{v_{H^2} - \tilde{v}_{H^2}}_{L^2(\Omega)} + k^{1/2} \bernkopfMelenkNorm{v_{H^2} - \tilde{v}_{H^2}}_{L^2(\partial \Omega)} \\
    & \qquad \lesssim \frac{1}{k}\left( \frac{k h}{p} + \left(\frac{k h}{p}\right)^2 \right) \bernkopfMelenkNorm{e^u}_{L^2(\Omega)}
    \lesssim \frac{h}{p} \bernkopfMelenkNorm{e^u}_{L^2(\Omega)},
  \end{align*}
  where the latter estimates are due to the boundedness of $\Omega$, $\sigma > 0$, and choosing $c_1$
  small and $c_2$ sufficiently large as well as elementary but tedious calculations.

	\noindent
  \textbf{Approximation of $\pmb{\psi}_A$}:
	To approximate $\pmb{\psi}_A$ we choose $\tilde{\pmb{\psi}}_{A} = \pmb{\Pi}^{\operatorname{div}, 2}_p \pmb{\psi}_A$ with $\pmb{\Pi}^{\operatorname{div}, 2}_p$ as in Corollary~\ref{bernkopf_melenk_corollary:global_RT_interpolation_operator_analytic} and apply the results therein.
	Furthermore we apply the estimates~\eqref{bernkopf_melenk_eq:decomposition_lemma_1} and~\eqref{bernkopf_melenk_eq:decomposition_lemma_3} of Lemma~\ref{bernkopf_melenk_lemma:duality_argument}.
  Proceeding as above together with a multiplicative trace inequality, again after tedious calculations, gives
  \begin{align*}
    & \bernkopfMelenkNorm{\nabla \cdot ( \pmb{\psi}_A - \tilde{\pmb{\psi}}_A )}_{L^2(\Omega)} +
    k \bernkopfMelenkNorm{ \pmb{\psi}_A - \tilde{\pmb{\psi}}_A }_{L^2(\Omega)} +
    k^{1/2} \bernkopfMelenkNorm{ ( \pmb{\psi}_A - \tilde{\pmb{\psi}}_A ) \cdot \pmb{n} }_{L^2(\partial \Omega)}
 \\ & \qquad
\lesssim \frac{h}{p} \bernkopfMelenkNorm{e^u}_{L^2(\Omega)}.
  \end{align*}

	\noindent
  \textbf{Approximation of $\pmb{\psi}_{H^2}$}:
  To approximate $\pmb{\psi}_{H^2}$ we choose $\tilde{\pmb{\psi}}_{H^2} = \pmb{\Pi}^{\operatorname{div}, 2}_p \pmb{\psi}_{H^2}$ with $\pmb{\Pi}^{\operatorname{div}, 2}_p$ as in Corollary~\ref{bernkopf_melenk_corollary:global_RT_interpolation_operator} and apply the results therein.
	We apply the estimate~\eqref{bernkopf_melenk_eq:decomposition_lemma_4} of Lemma~\ref{bernkopf_melenk_lemma:duality_argument}.
  Due to the multiplicative trace inequality we also have
  \begin{equation}
    \bernkopfMelenkNorm{ ( \pmb{\psi}_{H^2} - \tilde{\pmb{\psi}}_{H^2} ) \cdot \pmb{n} }_{L^2(\partial \Omega)} \leq \left(\frac{h}{p}\right)^{3/2} \bernkopfMelenkNorm{\pmb{\psi}_{H^2}}_{H^2(\Omega)}.
  \end{equation}
  Therefore we arrive at
  \begin{align*}
     & \bernkopfMelenkNorm{\nabla \cdot ( \pmb{\psi}_{H^2} - \tilde{\pmb{\psi}}_{H^2} )}_{L^2(\Omega)} +
     k \bernkopfMelenkNorm{\pmb{\psi}_{H^2} - \tilde{\pmb{\psi}}_{H^2} }_{L^2(\Omega)} +
     k^{1/2}\bernkopfMelenkNorm{ ( \pmb{\psi}_{H^2} - \tilde{\pmb{\psi}}_{H^2} ) \cdot \pmb{n} }_{L^2(\partial \Omega)} \\
     & \qquad \lesssim \frac{h}{p} \bernkopfMelenkNorm{\pmb{\psi}_{H^2}}_{H^2(\Omega)} \lesssim \frac{h}{p} \bernkopfMelenkNorm{e^u}_{L^2(\Omega)},
  \end{align*}
  where we used the estimate~\eqref{bernkopf_melenk_eq:decomposition_lemma_4} of Lemma~\ref{bernkopf_melenk_lemma:duality_argument}.
  Putting it all together we have
  \begin{align*}
    \bernkopfMelenkNorm{e^u}_{L^2(\Omega)}
    & \lesssim \frac{h}{p} (\bernkopfMelenkNorm{i k \pmb{e^\varphi} + \nabla e^u}_{L^2(\Omega)} + \bernkopfMelenkNorm{i k e^u + \nabla \cdot \pmb{e^\varphi}}_{L^2(\Omega)} + k^{1/2} \bernkopfMelenkNorm{ \pmb{e^\varphi} \cdot \pmb{n} + e^u}_{L^2(\partial\Omega)}) \\
    &\lesssim \frac{h}{p} \sqrt{b((\pmb{e^\varphi}, e^u), (\pmb{e^\varphi}, e^u))}.
  \end{align*}
  Applying again the Galerkin orthogonality and using the multiplicative trace inequality to absorb the term
  $k^{1/2} \bernkopfMelenkNorm{u - v_h}_{L^2(\partial \Omega)}$ into the $L^2$ norms of the volume yields the result.
  \qed
\end{proof}

We conclude this Section with a simple consequence of standard regularity theory and approximation properties of the employed finite element spaces in higher order Sobolev norms.

\begin{corollary}\label{bernkopf_melenk_corollary:higher_regularity_convergence_rate}
	For $s \geq 0$, $f \in H^s(\Omega)$ and $g \in H^{s+1/2}(\partial \Omega)$ we have
	$u \in H^{s+2}(\Omega)$,
	$u \in H^{s+3/2}(\partial \Omega)$,
	$\partial_n u \in H^{s+1/2}(\partial \Omega)$,
	$\pmb{\varphi} \in \pmb{H}^{s+1}(\Omega)$,
	$\nabla \cdot \pmb{\varphi} \in \pmb{H}^{s}(\Omega)$ and
	$\pmb{\varphi} \cdot \pmb{n} \in \pmb{H}^{s+1/2}(\partial \Omega)$.
	Furthermore there exist constants $c_1$, $c_2 > 0$ that are independent of $h$, $p$, and $k$ such that the conditions
  \begin{equation}
    \frac{kh}{p} \leq c_1 \qquad \text{and} \qquad p \geq c_2(\log k + 1)
  \end{equation}
  imply that the solution $(\pmb{\varphi}_h, u_h)$ satisfies
	\begin{equation*}
		\bernkopfMelenkNorm{u-u_h}_{L^2(\Omega)}
		\lesssim
		\left(\frac{h}{p}\right)^{s+1} (\bernkopfMelenkNorm{f}_{H^s(\Omega)} + \bernkopfMelenkNorm{g}_{H^{s+1/2}(\Omega)}),
	\end{equation*}
	for $ p \geq s $ with a wavenumber independent constant.
\end{corollary}

\begin{proof}
	\smartqed
	The first assertion follows immediately from standard regularity theory.
	Consider the case $s > 0$.
	Theorem~\ref{bernkopf_melenk_theorem:a_priori_estimate}
	together with a multiplicative trace inequality, which is applicable due
        to the already derived regularity of $\pmb{\varphi}$, gives
	\begin{align*}
		\bernkopfMelenkNorm{u-u_h}_{L^2(\Omega)}
		\lesssim
		\frac{h}{p}
		\Big(  &\bernkopfMelenkNorm{  u - v_h }_{H^1(\Omega)} +
						k \bernkopfMelenkNorm{ u - v_h }_{L^2(\Omega)} + \\
					 & \bernkopfMelenkNorm{  \pmb{\varphi} - \pmb{\psi}_h }_{H^1(\mathcal{T}_h)} +
						k \bernkopfMelenkNorm{ \pmb{\varphi} - \pmb{\psi}_h }_{L^2(\Omega)}
		\Big).
	\end{align*}
	Applying the higher order splitting of Theorem~\ref{bernkopf_melenk_proposition:higher_regularity_decomposition} and using the fact that $\pmb{\varphi} = i k^{-1} \nabla u$, one can easily estimate,
	as in the proof of Theorem~\ref{bernkopf_melenk_theorem:a_priori_estimate} together with the Corollaries~\ref{bernkopf_melenk_corollary:global_RT_interpolation_operator} and~\ref{bernkopf_melenk_corollary:global_RT_interpolation_operator_analytic},
	\begin{equation*}
		\bernkopfMelenkNorm{  \pmb{\varphi} - \pmb{\psi}_h }_{H^1(\Omega) } + k \bernkopfMelenkNorm{ \pmb{\varphi} - \pmb{\psi}_h }_{L^2(\Omega)}
		\lesssim  \left( \frac{h}{p} \right)^{s} ( \bernkopfMelenkNorm{f}_{H^s(\Omega)} + \bernkopfMelenkNorm{g}_{H^{s+1/2}(\Omega)} ).
	\end{equation*}
	Note the exponent $s$, since $\pmb{\varphi}$ is only in $\pmb{H}^{s+1}(\Omega)$.
	Furthermore, again as in the proof of Theorem~\ref{bernkopf_melenk_theorem:a_priori_estimate}, see also~\cite[Thm.~4.8]{melenk-parsania-sauter13}, we have
	\begin{equation*}
		\bernkopfMelenkNorm{  u - v_h }_{H^1(\Omega)} + k \bernkopfMelenkNorm{ u - v_h }_{L^2(\Omega)}
		\lesssim \left( \frac{h}{p} \right)^{s+1} ( \bernkopfMelenkNorm{f}_{H^s(\Omega)} + \bernkopfMelenkNorm{g}_{H^{s+1/2}(\Omega)} ),
	\end{equation*}
	now with the exponent $s+1$ since $u \in H^{s+2}(\Omega)$, which yields the result for $s>0$.
	In the case $s = 0$ one simply sets $v_h = 0$ as well as $\pmb{\psi}_h = 0$ and uses the wavenumber-explicit estimates of Theorem~\ref{bernkopf_melenk_proposition:higher_regularity_decomposition}.
	\qed
\end{proof}

\begin{remark}
	\smartqed
	Note that although we assume $f \in H^s(\Omega)$ and $g \in H^{s+1/2}(\partial \Omega)$ in Corollary~\ref{bernkopf_melenk_corollary:higher_regularity_convergence_rate},
	we only obtained a convergence rate $s+1$.
	This seems suboptimal when compared with classical FEM where,
	given sufficient regularity of the data and the geometry,
	one can expect a rate of $s+2$ for the convergence in the $L^2(\Omega)$-norm.
	Especially for $f \in L^2(\Omega)$ and $g \in H^{1/2}(\partial \Omega)$ one can only expect $h/p$
	for the FOSLS method compared to $h^2/p^2$ for the FEM. The proof of Corollary~\ref{bernkopf_melenk_corollary:higher_regularity_convergence_rate} is in that sense sharp since the leading error term in
	the {\sl a priori} estimate is
	\begin{equation*}
		\bernkopfMelenkNorm{ \nabla \cdot ( \pmb{\varphi} - \pmb{\psi}_h )}_{L^2(\Omega) } =
		\bernkopfMelenkNorm{ ik^{-1}f + iku - \nabla \cdot \pmb{\psi}_h }_{L^2(\Omega) },
	\end{equation*}
	where we used the fact $\pmb{\varphi} = i k^{-1} \nabla u$.
	The essential part is therefore to approximate an $f$ that is just in $L^2(\Omega)$ and therefore no further powers of $h$ can be gained.
	Assuming more regularity on $f$ would resolve this problem, however, the boundary data would restrict a further lifting of $\pmb{\varphi}$ in classical Sobolev spaces, but not in $H(\operatorname{div}, \Omega)$ spaces.
	This in turn would make it necessary to directly estimate $\bernkopfMelenkNorm{ \nabla \cdot ( \pmb{\varphi} - \pmb{\psi}_h )}_{L^2(\Omega) }$ instead of generously bounding it by $\bernkopfMelenkNorm{ \pmb{\varphi} - \pmb{\psi}_h }_{H^1(\mathcal{T}_h) }$.
	Last but not least there is the boundary term
	\begin{equation*}
		 \bernkopfMelenkNorm{ ( \pmb{\varphi} - \pmb{\psi}_h ) \cdot \pmb{n} }_{L^2(\partial \Omega)} =
		 \bernkopfMelenkNorm{ ik^{-1}g - u - \pmb{\psi}_h \cdot \pmb{n} }_{L^2(\partial \Omega)}.
	\end{equation*}
	Again if $g$ is only $H^{1/2}(\partial \Omega)$ one can only expect $\sqrt{h/p}$, but favorable in terms of $k$.
	\qed
\end{remark}

\section{Numerical examples}\label{bernkopf_melenk_section:numerical_examples}
All our calculations are performed with the $hp$-FEM code NETGEN/NGSOLVE
by J.~Sch\"oberl,~\cite{schoeberlNGSOLVE,schoeberl97}. We plot the error against $N_\lambda$,
the number of degrees of freedom per wavelength,
\begin{equation*}
  N_\lambda = \frac{2 \pi \sqrt[d]{\mathrm{DOF}}}{k \sqrt[d]{\mathrm{|\Omega|}}},
\end{equation*}
where the wavelength $\lambda$ and the wavenumber $k$ are
related via $k = 2 \pi / \lambda$ and $\mathrm{DOF}$ denotes the size of the linear system to be solved.
We compare the results of the classical FEM with the FOSLS method, measured in the relative
$L^2(\Omega)$ error. For the classical FEM we use the standard space $S_p(\mathcal{T}_h)$. For the FOSLS method we employ the pairing $\pmb{V}_h \times W_h = \pmb{\mathrm{BDM}}_p(\mathcal{T}_h) \times S_p(\mathcal{T}_h)$.

\begin{example}\label{bernkopf_melenk_example:smooth_solution}
  Let $\Omega$ be the unit circle in $\mathbb{R}^2$ and consider the problem
  \begin{equation*}
  	\begin{alignedat}{2}
  		- \Delta u - k^2 u &= 0 \qquad   &&\text{in } \Omega,\\
  		\partial_n u - i k u &= g        &&\text{on } \partial \Omega.
  	\end{alignedat}
  \end{equation*}
  where the data $g$ is such that the exact solution is given by $u(x,y) = \mathrm{e}^{i(k_1 x + k_2 y)}$ with $k_1 = - k_2 = \frac{1}{\sqrt{2}} k$.
	For the numerical studies, this problem will be solved using $h$-FEM and $h$-FOSLS with polynomial degrees $p = 1,2,3,4$.
	The results are visualized in Figure~\ref{bernkopf_melenk_figure:h_fem_versus_h_fosls}.
	For both methods we observe the expected convergence $O(h^{p+1})$ in the relative $L^2(\Omega)$ error.
	Note that for both methods higher order versions are less prone to the pollution effect.
	At the same number of degrees of freedom per wavelength we also observe that the classical FEM is superior to FOSLS, when measured in achieved accuracy in $L^2(\Omega)$.
	This is not surprising since, for the same mesh and polynomial degree $p$, the number of degrees of freedom
of the FOSLS is roughly three times as large as for the classical FEM.
	Note, however, that we do not consider any solver aspects of the employed methods, where FOSLS might have advantages over the classical FEM since its system matrix is positive definite.
\end{example}

\begin{figure}[h]
	\centering
	\includegraphics[width=\textwidth]{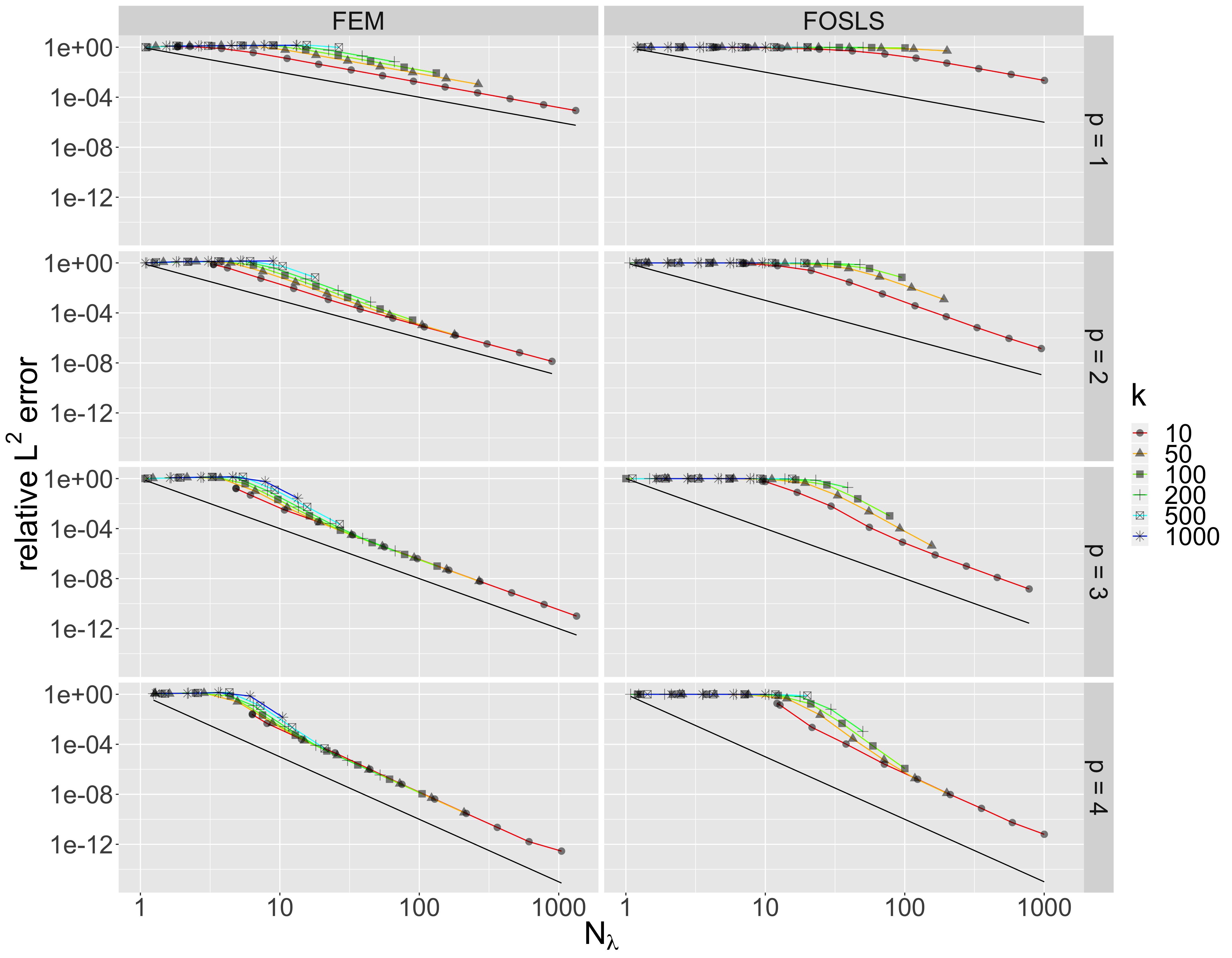}
	\caption{
		Comparison between the $h$-FEM (left) and $h$-FOSLS (right) for $p = 1$, $2$, $3$, $4$ as described in Example~\ref{bernkopf_melenk_example:smooth_solution}.
		The reference line in black corresponds to $h^{p+1}$.
	}
	\label{bernkopf_melenk_figure:h_fem_versus_h_fosls}
\end{figure}

\begin{example}\label{bernkopf_melenk_example:singular_solution_corner_domain}
	For $\pi < \omega< 2 \pi$ let $\Omega = \{ (r \cos \varphi, r \sin \varphi) \colon r \in (0,1), \varphi \in (0, \omega) \} \subset\mathbb{R}^2$ and consider
  \begin{equation*}
  	\begin{alignedat}{2}
  		- \Delta u - k^2 u &= 0 \qquad   &&\text{in } \Omega,\\
  		\partial_n u - i k u &= g        &&\text{on } \partial \Omega.
  	\end{alignedat}
  \end{equation*}
  The data $g$ is such that the exact solution is given by $u(x,y) = J_{\alpha}(kr) \cos(\alpha \varphi)$, with $\alpha = 3 \pi / 2$.
	Standard regularity theory gives $u \in H^{1 + \alpha - \varepsilon}(\Omega)$ for every $\varepsilon > 0$.
	In the numerical experiments we keep $kh = 5$ and perform a $p$-FEM and a $p$-FOSLS method up to $p = 46$ and $p = 29$, respectively.
	The results are visualized in Figure~\ref{bernkopf_melenk_figure:p_fem_versus_p_fosls}. We observe that the FEM has
  significantly smaller errors than the FOSLS. For a discussion of the expected $L^2(\Omega)$-convergence rates
  of the $p$-FEM, we refer the reader to~\cite[Remark after Thm.~3 and Section~3]{jensen-suri92}.
\end{example}

\begin{figure}[h]
	\centering
	\includegraphics[width=\textwidth]{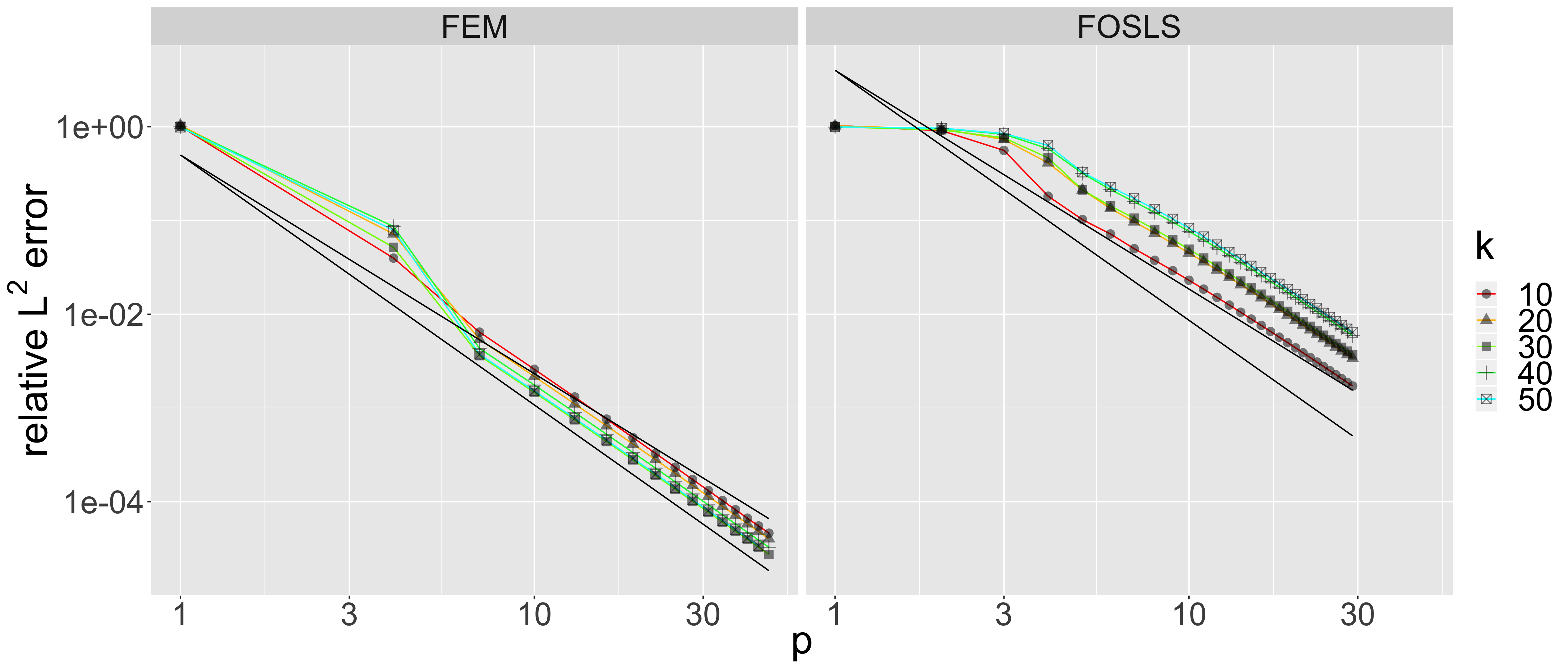}
	\caption{
		Comparison between the $p$-FEM (left) and $p$-FOSLS (right) for $kh = 5$ as described in Example~\ref{bernkopf_melenk_example:singular_solution_corner_domain}.
		We include the reference lines $p^{-4 \cdot 2/3} = p^{-8/3}$ and $p^{-(2 \cdot 2/3 + 1)} = p^{-7/3}$.
	}
	\label{bernkopf_melenk_figure:p_fem_versus_p_fosls}
\end{figure}
The next example focuses on the Helmholtz equation with right-hand side $f$ with finite Sobolev regularity.
\begin{example}\label{bernkopf_melenk_example:singular_solution_singular_f}
	Let $\Omega = (-1, 1) \subset \mathbb{R}$ and $f = -\chi_{(-1,0]} + \chi_{(0,1)}$, where $\chi_A$ denotes the indicator function on $A \subset \mathbb{R}$.
	The function $f$ is in $H^{1/2 - \varepsilon}(\Omega)$ for every $\varepsilon > 0$.
	We consider uniform meshes $\mathcal{T}_h$ on $\Omega$ such that the break point zero is \emph{not} a node, as otherwise the piecewise smooth solution could be approximated very well.  We study
  \begin{equation*}
  	\begin{alignedat}{2}
  		- u^{\prime \prime} - k^2 u &= f \qquad   &&\text{in } \Omega,\\
  		\partial_n u - i k u &= g        &&\text{on } \partial \Omega.
  	\end{alignedat}
  \end{equation*}
  where the data $g$ is such that the exact solution is given by
	\[ u(x) =
		\begin{cases}
      \cos(kx) + \frac{1}{k^2}   & x \leq 0 \\
      (1 + \frac{2}{k^2}) \cos(kx) - \frac{1}{k^2} & x > 0
   	\end{cases}
	\]
	Standard regularity theory gives $u \in H^{2.5 - \varepsilon}(\Omega)$ for every $\varepsilon > 0$.
	For the $h$-FEM we expect $O(h^{\min(2 + 0.5, p+1)})$.
	In fact for $p>1$ one can show (cf.~\cite[Cor.~{4.6}]{esterhazy-melenk14}) that
  			$k \bernkopfMelenkNorm{u-u_h^{\mathrm{FEM}}}_{L^2(\Omega)} \lesssim h^{2.5}$ and, by inspection,
  			$\bernkopfMelenkNorm{u}_{L^2(\Omega)} = O(1)$ (uniformly in $k$).
	It is therefore expedient to plot $k^{3.5} \bernkopfMelenkNorm{u-u_h^{\mathrm{FEM}}}_{L^2(\Omega)}/\|u\|_{L^2(\Omega)}$
        versus $N_\lambda \sim (kh)$.
	For the $h$-FOSLS Corollary~\ref{bernkopf_melenk_corollary:higher_regularity_convergence_rate} predicts only
        $O(h^{\min(1+0.5, p+1)})$.
	The numerical results show, however, for both methods convergence $O(h^{\min(2.5, p+1)})$.
	The results are visualized in Figure~\ref{bernkopf_melenk_figure:h_fem_versus_h_fosls_singular_f}.
\end{example}

\begin{figure}[h]
	\centering
	\includegraphics[width=\textwidth]{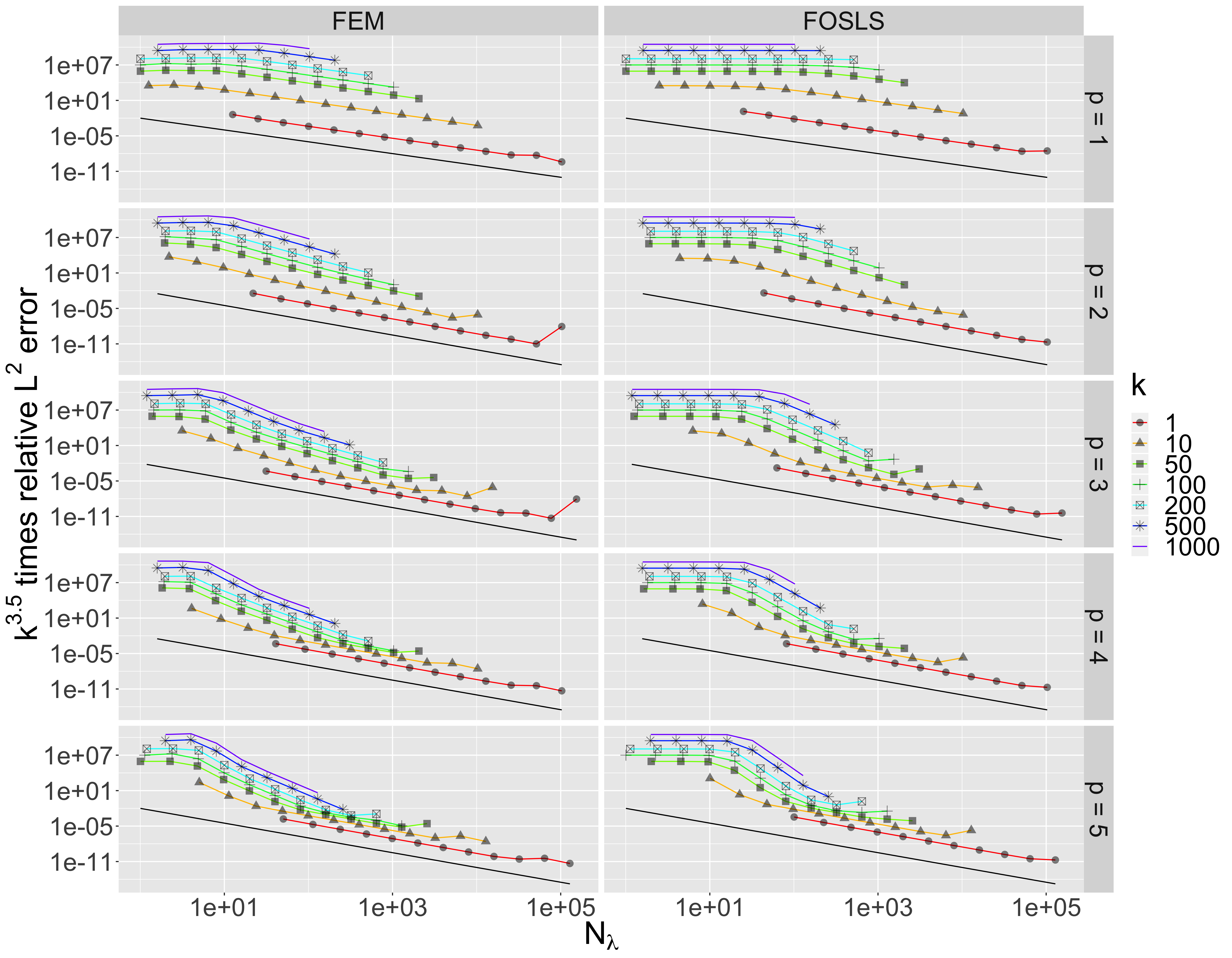}
	\caption{
	Comparison between the $h$-FEM (left) and $h$-FOSLS (right) for $p = 1,\ldots,5$ as described in Example~\ref{bernkopf_melenk_example:singular_solution_singular_f}.
	The reference line in black corresponds to $h^{\min(2.5, p+1)}$.
	}
	\label{bernkopf_melenk_figure:h_fem_versus_h_fosls_singular_f}
\end{figure}

\begin{remark}\label{bernkopf_melenk_remark:error_terms_explaining_imporved_convergence}
	The numerical results of Example~\ref{bernkopf_melenk_example:singular_solution_singular_f}
	visualized in Figure~\ref{bernkopf_melenk_figure:h_fem_versus_h_fosls_singular_f}
	indicate that Corollary~\ref{bernkopf_melenk_corollary:higher_regularity_convergence_rate} is in fact suboptimal as
	it predicts only a  convergence $O(h^{1.5})$ while we observe $O(h^{\min(2.5, p+1)})$.
        A starting point for understanding this
        better convergence behavior could be two observations: first, the duality argument in
        Theorem~\ref{bernkopf_melenk_theorem:a_priori_estimate} is based on the regularity
        $(\pmb{\psi},v) \in \pmb{H}^2(\Omega) \times H^2(\Omega)$ of the dual solution
        $(\pmb{\psi},v)$ whereas in fact (see the proof of Lemma~\ref{bernkopf_melenk_lemma:duality_argument})
        $(\pmb{\psi},v) \in \pmb{H}^2(\operatorname{div},\Omega) \times H^2(\Omega)$. Second, a more careful
        application of the Cauchy-Schwarz inequality~\eqref{bernkopf_melenk_eq:first_cs_estimate_in_a_priori_estimate}
	at the beginning of the proof of Theorem~\ref{bernkopf_melenk_theorem:a_priori_estimate} is advisable. In this connection,
        we point to the fact that the terms in the square brackets in~\eqref{bernkopf_melenk_eq:first_cs_estimate_in_a_priori_estimate} are not of the same order.
        To illustrate this, we plot the components
	\begin{equation}
		e_1 \coloneqq i k \pmb{e^\varphi} + \nabla e^u \quad \text{ and } \quad e_2 \coloneqq i k e^u + \nabla \cdot \pmb{e^\varphi}
	\end{equation}
	in Figure~\ref{bernkopf_melenk_figure:error_components_with_singular_f} for the problem studied in Example~\ref{bernkopf_melenk_example:singular_solution_singular_f}.
\qed
\end{remark}

\begin{figure}[h]%
    \centering
		{\includegraphics[width=5cm]{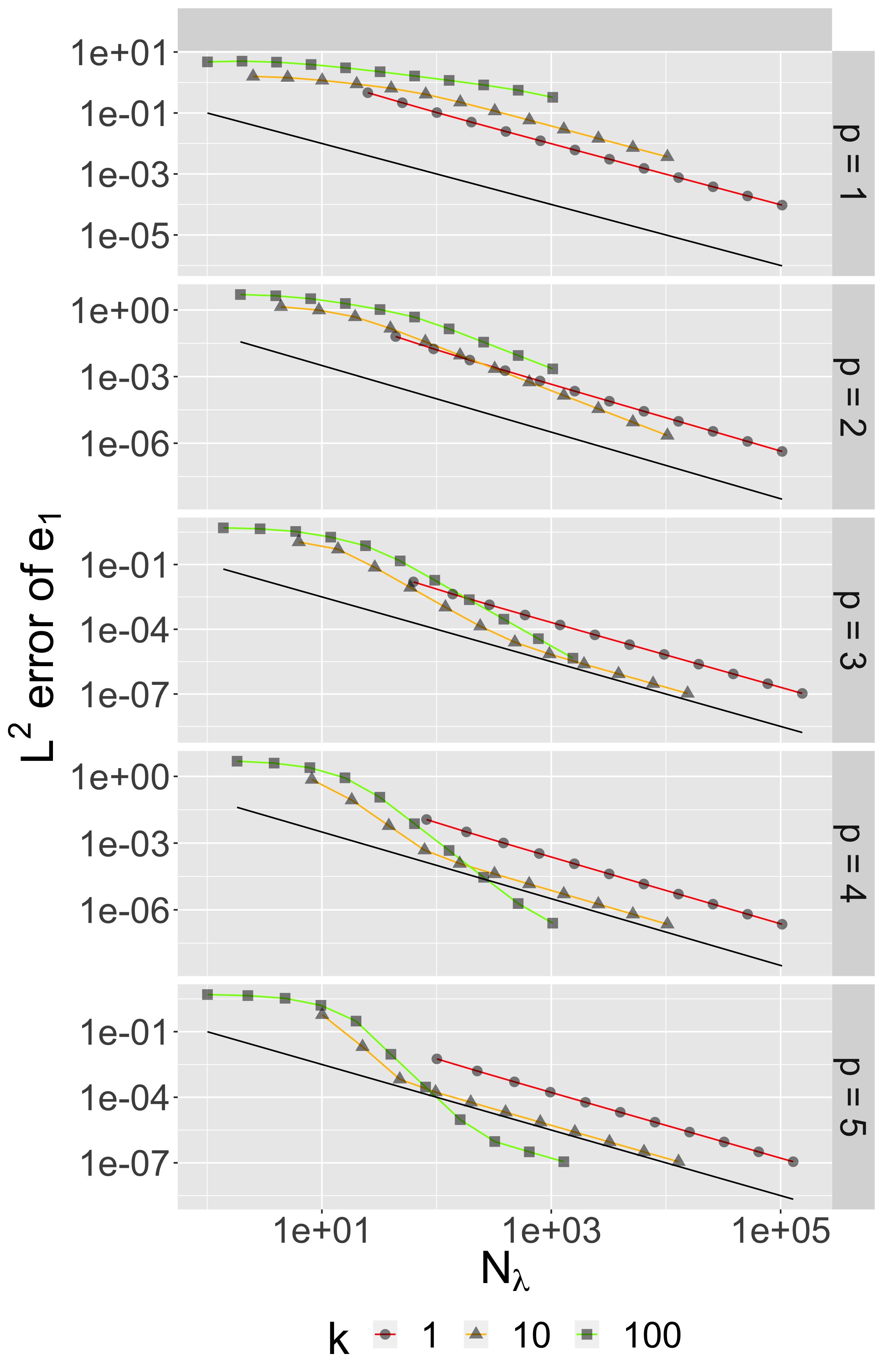} }%
    \qquad
		{\includegraphics[width=5cm]{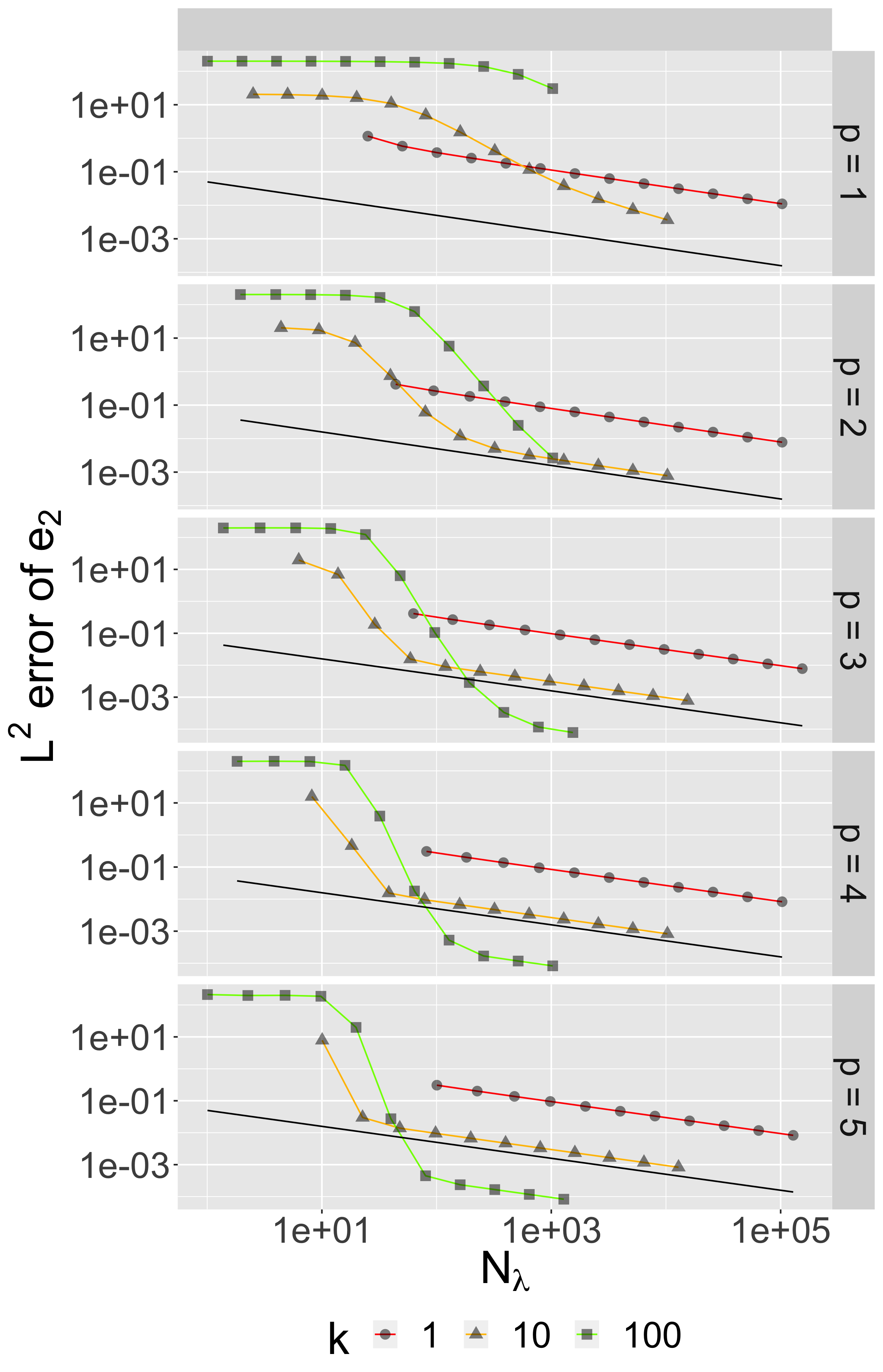} }%
		\caption{
		Comparison between the error terms
		$e_1 \coloneqq \bernkopfMelenkNorm{i k \pmb{e^\varphi} + \nabla e^u}_{L^2(\Omega)}$ (left) and
		$e_2 \coloneqq \bernkopfMelenkNorm{i k e^u + \nabla \cdot \pmb{e^\varphi}}_{L^2(\Omega)}$ (right)
		for $p = 1,\ldots,5$ as described in Remark~\ref{bernkopf_melenk_remark:error_terms_explaining_imporved_convergence}
		and Example~\ref{bernkopf_melenk_example:singular_solution_singular_f}.
		The reference line on the left corresponds to $h^1$ for $p = 1$ and $h^{1.5}$ for $p > 1$.
		The reference line on the right corresponds to $h^{1/2}$.
		}%
		\label{bernkopf_melenk_figure:error_components_with_singular_f}%
\end{figure}
\noindent
\textbf{Acknowledgement}:
MB is grateful for the financial support by the Austrian Science Fund (FWF) through the doctoral school \textit{Dissipation and dispersion in nonlinear PDEs} (grant W1245).
MB thanks the workgroup of Joachim Sch\"oberl (TU Wien) for help regarding the numerical experiments.


\bibliographystyle{spmpsci}
\bibliography{literature.bib}

\end{document}